\documentclass[a4paper,11pt]{article}

\usepackage[utf8]{inputenc}
\usepackage[T1]{fontenc}
\usepackage[english]{babel}

\usepackage[tbtags]{amsmath}
\usepackage{amssymb}
\usepackage{amsfonts}
\usepackage{amsthm}
\usepackage{calrsfs}
\usepackage{array}
\usepackage{bbm}
\usepackage{stmaryrd}
\usepackage{upgreek}

\usepackage{mathtools}
\usepackage[svgnames]{xcolor}
\usepackage{hyperref}

\renewcommand{\aa}{{m}}
\newcommand{\bb}{{\overline m}}
\newcommand\cc{{\cal c}}
\newcommand{\ds}{\displaystyle}
\newcommand{\e}{\varepsilon}

\renewcommand{\H}{{\cal H}}
\renewcommand\iint{\displaystyle\int_{-1}^1}

\renewcommand{\k}{k}

\renewcommand{\ll}{\hat L}

\newcommand{\m}[1]{\mathbbm{#1}}
\newcommand{\pnu}{\partial_\nu}
\newcommand{\pp}{\Pi}
\newcommand{\ps}{{\partial_s}}
\newcommand{\py}{{\partial_y}}
\newcommand{\q}[1]{\mathcal{#1}}

\newcommand\RR{{\cal R}}
\renewcommand\SS{{\cal S}}

\newcommand{\vc}[2]{\begin{pmatrix} #1\\#2\end{pmatrix}}

\DeclareMathOperator{\Id}{\mathrm{Id}}

\theoremstyle{plain}
\newtheorem{thm}{Theorem}
\newtheorem*{thm*}{Theorem}

\newtheorem{prop}{Proposition}[section]
\newtheorem{cor}[prop]{Corollary}
\newtheorem{lem}[prop]{Lemma}
\newtheorem{defi}[prop]{Definition}

\theoremstyle{definition}

\theoremstyle{remark}
\newtheorem*{nb}{Remark}
\newtheorem*{claim}{Claim}

\makeatletter
\def\blfootnote{\xdef\@thefnmark{}\@footnotetext}
\makeatother

\title{\bf Construction of a multi-soliton blow-up solution to the semilinear wave equation in one space dimension}
\date{April 24, 2012}
\author{Raphaël Côte\\
{\it \small CNRS UMR 7640 CMLS \'Ecole Polytechnique}\\
Hatem Zaag\\
{\it \small CNRS UMR 7539 LAGA Universit\'e Paris 13}}

\allowdisplaybreaks

\begin{document}

\maketitle   

\begin{abstract}
We consider the semilinear wave equation with power nonlinearity in one space dimension. Given a blow-up solution with a characteristic point, we refine the blow-up behavior first derived by Merle and Zaag. We also refine the geometry of the blow-up set near a characteristic point, and show that except may be for one exceptional situation, it is never symmetric with the respect to the characteristic point. Then, we show that all blow-up modalities predicted by those authors do occur. More precisely, given any integer $k\ge 2$ and $\zeta_0 \in \m R$, we construct a blow-up solution with a characteristic point $a$, such that the asymptotic behavior of the solution near $(a,T(a))$ shows 
a decoupled sum of $k$ solitons with alternate signs, whose centers (in the hyperbolic geometry) have $\zeta_0$ as a center of mass, for all times.
\end{abstract}

\medskip



\medskip


\section{Introduction}
We consider the one-dimensional semilinear wave equation
\begin{equation}\label{eq:nlw_u}
\left\{
\begin{array}{l}
\partial^2_{t} u =\partial^2_{x} u+|u|^{p-1}u,\\
u(0)=u_0\mbox{ and }\partial_t u(0)=u_1,
\end{array}
\right.
\end{equation}
where $u(t):x\in{\m R} \rightarrow u(x,t)\in{\m R}$, $p>1$, $u_0\in \rm H^1_{\rm loc,u}$
and $u_1\in \rm L^2_{\rm loc,u}$ with
 $\|v\|_{\rm L^2_{\rm loc,u}}^2=\ds\sup_{a\in {\m R}}\int_{|x-a|<1}|v(x)|^2dx$ and $\|v\|_{\rm H^1_{\rm loc,u}}^2 = \|v\|_{\rm L^2_{\rm loc,u}}^2+\|\nabla v\|_{\rm L^2_{\rm loc,u}}^2$.\\
We solve equation \eqref{eq:nlw_u} locally in time in the space ${\rm H}^1_{\rm loc,u}\times {\rm L}^2_{\rm loc, u}$ (see Ginibre, Soffer and Velo \cite{GSVjfa92}, Lindblad and Sogge \cite{LSjfa95}).
For the existence of blow-up solutions, we have the following blow-up criterion from Levine \cite{Ltams74}: If $(u_0,u_1)\in H^1\times L^2({\m R})$ satisfies
\[
\int_{{\m R}}\left(\frac 12 |u_1(x)|^2+\frac 12|\partial_x u_0(x)|^2-\frac 1{p+1}|u_0(x)|^{p+1}\right)dx<0,\label{levine}
\]
then the solution of \eqref{eq:nlw_u} cannot be global in time. More blow-up results can be found in
Caffarelli and Friedman \cite{CFtams86, CFarma85},
Alinhac \cite{Apndeta95, Afle02} and Kichenassamy and Littman \cite{KL1cpde93, KL2cpde93}.

\bigskip

If $u$ is an arbitrary blow-up solution of \eqref{eq:nlw_u}, we define (see for example Alinhac \cite{Apndeta95}) a 1-Lipschitz curve $\Gamma=\{(x,T(x))\}$ 
such that the maximal influence domain $D$ of $u$ (or the domain of definition of $u$) is written as 
\begin{equation}\label{defdu}
D=\{(x,t)\;|\; t< T(x)\}.
\end{equation}
$\bar T=\inf_{x\in {\m R}}T(x)$ and $\Gamma$ are called the blow-up time and the blow-up graph of $u$. 
A point $x_0$ is a non characteristic point 
if
\begin{equation}\label{nonchar}
\mbox{there are }\delta_0\in(0,1)\mbox{ and }t_0<T(x_0)\mbox{ such that }
u\;\;\mbox{is defined on }{\mathcal C}_{x_0, T(x_0), \delta_0}\cap \{t\ge t_0\}
\end{equation}
where ${\cal C}_{\bar x, \bar t, \bar \delta}=\{(x,t)\;|\; t< \bar t-\bar \delta|x-\bar x|\}$. We denote by $\RR$ (resp. $\SS$) the set of non characteristic (resp. characteristic) points.

\bigskip

In order to study the asymptotic behavior of $u$ near a given $(x_0,T(x_0))\in \Gamma$, it is convenient to introduce similarity variables defined for all $x_0\in\m R$ and $ T_0\in \m R$ by
\begin{equation} \label{def:w}
w_{x_0, T_0}(y,s) = (T_0 -t)^{\frac{2}{p-1}} u(x,t), \quad y = \frac{x-x_0}{T_0-t}, \quad s = - \log (T_0-t).
\end{equation}
If $T_0=T(x_0)$, we will simply write $w_{x_0}$ instead of $w_{x_0, T(x_0)}$. 
The function $w=w_{x_0}$ satisfies the following equation for all $y\in(-1,1)$ and $s\ge - \log T(x_0)$: 
\begin{gather} \label{eq:nlw_w}
\partial^2_sw = \q L w - \frac{2(p+1)}{(p-1)^2} w + |w|^{p-1} w - \frac{p+3}{p-1}\partial_s w - 2y \partial^2_{y,s}w, \\
\text{where} \quad
\q L w = \frac{1}{\rho} \partial_y (\rho (1-y^2) \py w) 
\quad \text{and} \quad \rho = (1-y^2)^{\frac{2}{p-1}}.
\end{gather}
This equation can be put in the following first order form:
\begin{equation} \label{eq:nlw_wws}
\partial_s \begin{pmatrix} w_1 \\ w_2 \end{pmatrix} =
\begin{pmatrix} w_2 \\ \displaystyle \q L w_1 - \frac{2(p+1)}{(p-1)^2} w_1 + |w_1|^{p-1} w_1 - \frac{p-3}{p-1} w_2 - 2y \py w_2 \end{pmatrix}.
\end{equation}
The Lyapunov functional for equation \eqref{eq:nlw_w}
\begin{equation}\label{defenergy}
E(w(s))= \iint \left(\frac 12 \left(\partial_s w\right)^2 + \frac 12  \left(\partial_y w\right)^2 (1-y^2)+\frac{(p+1)}{(p-1)^2}w^2 - \frac 1{p+1} |w|^{p+1}\right)\rho dy
\end{equation}
is defined for $(w,\partial_s w) \in \H$ where  
\begin{equation}\label{defnh0}
\H = \left\{(q_1,q_2) 
\;\; \middle| \;\;\|(q_1,q_2)\|_{\H}^2\equiv \int_{-1}^1 \left(q_1^2+\left(q_1'\right)^2  (1-y^2)+q_2^2\right)\rho dy<+\infty\right\}.
\end{equation}
We also introduce the projection of the space $\H$ \eqref{defnh0} on the first coordinate:
\[
\H_0 = \left\{r\in H^1_{loc}\;\;\middle|\;\;\|r\|_{\H_0}^2\equiv \int_{-1}^1 \left(r^2+\left(r'\right)^2  (1-y^2)\right)\rho dy<+\infty\right\}.
\]
Finally, we introduce for all $|d|<1$ the following stationary solutions of \eqref{eq:nlw_w} (or solitons) defined by 
\begin{equation}\label{defkd}
\kappa(d,y)=\kappa_0 \frac{(1-d^2)^{\frac 1{p-1}}}{(1+dy)^{\frac 2{p-1}}}\mbox{ where }\kappa_0 = \left(\frac{2(p+1)}{(p-1)^2}\right)^{\frac 1{p-1}} \mbox{ and }|y|<1.
\end{equation}

In \cite{MZjfa07, MZcmp08, MZajm11, MZisol10} (see also the note \cite{MZxedp10}), Merle and Zaag gave an exhaustive description of the geometry of the blow-up set on the one hand, and the asymptotic behavior of solutions near the blow-up set on the other hand (they also extended their results to the radial case with conformal or subconformal power nonlinearity outside the origin in \cite{MZbsm11}):

\medskip

\noindent{\it - Geometry of the blow-up set}: In \cite[Theorem 1]{MZcmp08} (and the following remark),  and \cite[Theorems 1 and 2]{MZisol10} (and the following remark), the following is proved:

\medskip

{\label{old}\it (i) $\RR$ is a non empty open set, and $x\mapsto T(x)$ is of class $C^1$ on $\RR$;

(ii) $\SS$ is made of isolated points, and given $x_0\in\SS$, if $0<|x-x_0|\le \delta_0$, then
\begin{equation}\label{chapeau0}
\frac{|x-x_0|}{C_0|\log(x-x_0)|^{\frac{(k(x_0)-1)(p-1)}2}}\le  T(x)- T(x_0)+|x-x_0| \le \frac{C_0|x-x_0|}{|\log(x-x_0)|^{\frac{(k(x_0)-1)(p-1)}2}}
\end{equation}
for some $\delta_0>0$ and $C_0>0$, where $k(x_0)\ge 2$ is an integer. Moreover, estimate \eqref{chapeau0} remains true after differentiation.
}

\medskip

{\it - Classification of the blow-up behavior near the singularity in $(x_0, T(x_0))$}\label{classification}: From \cite[Corollary 4 page 49, Theorem 3 page 48]{MZjfa07}, \cite[Theorem 1 page 58]{MZcmp08} and \cite[Theorem 6]{MZajm11}, we recall the asymptotic behavior of $u(x,t)$ near the singular point $(x_0, T(x_0))$, according to the fact that $x_0$ is a non-characteristic point or not:

\medskip

{\it
(iii) 
There exist $\mu_0>0$ and $C_0>0$ such that for all $x_0\in\RR$, there exist $\theta(x_0)=\pm 1$ and $s_0(x_0)\ge - \log T(x_0)$ such that for all $s\ge s_0$:
\begin{equation}\label{profile}
\left\|\vc{w_{x_0}(s)}{\partial_s w_{x_0}(s)}-\theta(x_0)\vc{\kappa(T'(x_0))}{0}\right\|_{\H}\le C_0 e^{-\mu_0(s-s_0)}.
\end{equation}
Moreover, $E(w_{x_0}(s)) \to E(\kappa_0)$ as $s\to \infty$.

(iv) \label{old4}
If $x_0\in\SS$, then it holds that
\begin{equation}\label{cprofile00}
\left\|\vc{w_{x_0}(s)}{\ps w_{x_0}(s)} - \theta_1\vc{\ds\sum_{i=1}^{k(x_0)} (-1)^{i+1}\kappa(d_i(s))}0\right\|_{\H} \to 0\mbox{ and }E(w_{x_0}(s))\to k(x_0)E(\kappa_0)
\end{equation}
as $s\to \infty$, where the integer $k(x_0)\ge 2$ has been defined in \eqref{chapeau0}, for some $\theta_1=\pm 1$
and continuous $d_i(s)=-\tanh \zeta_i(s)\in (-1,1)$ for $i=1,...,k(x_0)$. Moreover, for some $C_0>0$, for all $i=1,...,k(x_0)$ and $s$ large enough, we have
\begin{equation}\label{equid00}
\left|\zeta_i(s)-\left(i-\frac{(k(x_0)+1)}2\right)\frac{(p-1)}2\log s\right|\le C_0. 
\end{equation}
}

\medskip

Our first result refines the expansion \eqref{equid00} up to the order $o(1)$. Introduce
\begin{equation}\label{solpart}
\bar \zeta_i(s) = \left(i-\frac{(k+1)}2\right)\frac{(p-1)}2\log s + \bar\alpha_i(p,k)
\end{equation}
where the sequence $(\bar\alpha_i)_{i=1,\dots,k}$ is uniquely determined by the fact that $(\bar \zeta_i(s))_{i=1,\dots,k}$ is an explicit solution with zero center of mass for the following ODE system:
\begin{equation} \label{eq:tl}
 \frac 1{c_1}\dot \zeta_i = e^{ - \frac{2}{p-1} (\zeta_i - \zeta_{i-1}) } - e^{- \frac{2}{p-1} (\zeta_{i+1} - \zeta_i) }, 
\end{equation}
where $c_1=c_1(p)>0$ and $\zeta_0(s)\equiv \zeta_{k+1}(s) \equiv 0$ (see \eqref{defalphai} below for a proof of this fact). Note that $c_1=c_1(p)>0$ is the constant appearing in system \eqref{eqz}, itself inherited from Proposition 3.2 of \cite{MZajm11}.
With this definition, we can state our first result:
\begin{thm}[Refined asymptotics near a characteristic point] \label{propref}
Consider $u(x,t)$ a blow-up solution of equation \eqref{eq:nlw_u} and $x_0$ a characteristic point with $k(x_0)$ solitons. Then, there is 
$\zeta_0(x_0)\in \m R$ 
such that estimate \eqref{cprofile00} holds with
\begin{equation}\label{refequid}
d_i(s) = -\tanh \zeta_i(s) \quad \text{and} \quad \zeta_i(s) = \bar \zeta_i(s) + \zeta_0,
\end{equation}
where $\bar \zeta_i(s)$ is introduced above in \eqref{solpart}. 
\end{thm}
\begin{nb}
As one can see from \eqref{refequid} and \eqref{solpart}, $\zeta_0$ is the center of mass of the $\zeta_i(s)$ for any $s\ge - \log T(x_0)$.
\end{nb}
\begin{nb}
Following the analysis of Merle and Zaag in \cite{MZbsm11}, our result holds with the same proof for the higher-dimensional radial case
\begin{equation}\label{equr}
\partial_t^2 u = \partial_r^2 u +\frac{(N-1)}r \partial_ ru +|u|^{p-1}u,
\quad \text{with} \quad 
p\le 1+\frac 4{N-1} \mbox{ if }N\ge 2,
\end{equation}
provided that we consider a characteristic point different from the origin.
\end{nb}

Theorem \ref{propref} enables us to refine estimate \eqref{chapeau0} proved in \cite{MZisol10} for $T(x)$ and $T'(x)$ when $x$ is near a characteristic point. More precisely, we have the following:
\begin{cor}[Refined behavior for the blow-up set near a characterstic point]\label{corref}
Consider $u(x,t)$ a blow-up solution of equation \eqref{eq:nlw_u} and $x_0$ a characteristic point with $k(x_0)$ solitons and $\zeta_0(x_0)\in \m R$ as center of mass of the solitons' center as shown in \eqref{cprofile00} and \eqref{refequid}. Then, 
\begin{align}
T'(x)& = -\theta(x)\left(1-\frac{\gamma e^{-2\theta(x)\zeta_0(x_0)}(1+o(1))}{|\log|x-x_0||^{\frac{(k(x_0)-1)(p-1)}2}}\right)\label{cor1}\\
T(x)&=T(x_0)-|x-x_0|+\frac{\gamma e^{-2\theta(x)\zeta_0(x_0)}|x-x_0|(1+o(1))}{|\log|x-x_0||^{\frac{(k(x_0)-1)(p-1)}2}}\label{cor0}
\end{align}
as $x\to x_0$, where $\theta(x) = \frac{x-x_0}{|x-x_0|}$ and $\gamma=\gamma(p)>0$.
\end{cor}
\begin{nb}Unlike what one may think from the less accurate estimate \eqref{chapeau0}, we surprisingly see from this corollary that the blow-up set is {\it never} symmetric with respect to a characteristic point $x_0$, except maybe when $\zeta_0(x_0)=0$.
\end{nb}
\begin{nb} As usual in blow-up problems, the geometrical features of the blow-up set (here, $T(x)$ and $T'(x)$) are linked to the parameters of the asymptotic behavior of the solution (here, $k(x_0)$ the number of solitons in similarity variables and $\zeta_0(x_0)$, the location of their center of mass).
\end{nb}

Following this classification given for arbitrary blow-up solutions, we asked the question whether all these blow-up modalities given above in {\it (iii)} and {\it (iv)} (refined by Theorem \ref{propref}) do occur or not.

\medskip

As far as non-characteristic points are concerned, the answer is easy.
 Indeed, any blow-up solution (for example those constructed by Levine's criterion given on page \pageref{levine}) has non-characteristic points, as stated above in fact {\it (i)} of page \pageref{old} .\\
 Regarding the asymptotic behavior, any profile given in \eqref{profile} does occur. Indeed, note first that for any $d\in(-1,1)$, the function 
\begin{equation}\label{particular}
u(x,t)= (1-t)^{-\frac 2{p-1}}\kappa\left(d,\frac x{1-t}\right)=\frac{\kappa_0(1-d^2)^{\frac 1{p-1}}}{(1-t +dx)^{\frac 2{p-1}}}
\end{equation}
is a particular solution to equation \eqref{eq:nlw_u} defined for all $(x,t) \in \m R^2$ such that $1-t +dx>0$, blowing up on the curve $T(x) = 1+dx$ and such that for any $x_0\in {\m R}$, $T'(x_0)=d$ and $w_{x_0}(y,s) = \kappa(d,y)=\kappa(T'(x_0),y)$, and \eqref{profile} is trivially true. However, this is not a solution of the Cauchy problem at $t=0$, in the sense that it is not even defined for all $x \in{\m R}$ when $t=0$. This is in fact not a problem thanks to the finite speed of propagation. Indeed, performing a truncation of \eqref{particular} at $t=0$, the new solution will coincide with \eqref{particular} for all $|x_0|\le R$ and $t\in[0,T(x_0))$ for some $R>0$, and \eqref{profile} holds for the new solution as well, for all $|x_0|<R$.

\medskip

Now, for characteristic points, the answer is much more delicate. Unlike what was commonly believed after the work of Caffarelli and Friedman \cite{CFtams86, CFarma85}, Merle and Zaag proved in\cite[Proposition 1]{MZajm11} the {\it existence of solutions} of \eqref{eq:nlw_u} such that
\[
\SS\neq \emptyset.
\]
Since that solution was odd by construction, the number of solitons appearing in the decomposition \eqref{cprofile00} has to be even. No other information on the number of solitons was available. After this result, the following question remained open :

\medskip

{\it Given an integer $k\ge 2$, is there a blow-up solution of equation \eqref{eq:nlw_u} with a characteristic point $x_0$ such that the decomposition \eqref{cprofile00} holds with $k$ solitons?}

\medskip

In this paper, we show that the answer is {\it yes}, and we do better, by prescribing the location of the center of mass of the $\zeta_i(s)$ in \eqref{equid00}. This is our second result:
\begin{thm}[Existence of a solution with prescribed blow-up behavior at a characteristic point]\label{mainth} For any integer $k\ge 2$ and $\zeta_0\in \m R$, there exists a blow-up solution $u(x,t)$ to equation \eqref{eq:nlw_u} in $\rm H^1_{\rm loc,u}\times \rm L^2_{\rm loc,u}(\m R)$ with $0\in\SS$ such that 
\begin{equation}\label{cprofile0}
\left\|\vc{w_0(s)}{\ps w_0(s)} - \vc{\ds\sum_{i=1}^{k} (-1)^{i+1}\kappa(d_i(s))}0\right\|_{\H} \to 0\mbox{ as }s\to \infty,
\end{equation}
with
\begin{equation}\label{refequid1}
d_i(s) = -\tanh \zeta_i(s), \quad
\zeta_i(s) = \bar \zeta_i(s) + \zeta_0
\end{equation}
and $\bar \zeta_i(s)$ defined in \eqref{solpart}.
\end{thm}
\begin{nb}
Note from \eqref{refequid1} and \eqref{solpart} that the barycenter of $\zeta_i(s)$ is fixed, in the sense that
\begin{equation}\label{barycenter}
\frac{\zeta_1(s)+ \dots +\zeta_k(s)}k= \frac{\bar\zeta_1(s)+ \dots +\bar\zeta_k(s)}k+\zeta_0=\zeta_0,\;\;\forall s\ge -\log T(0).
\end{equation}
\end{nb}
\begin{nb} Note that this result uses our argument for Theorem \ref{propref}, in particular our analysis of ODE \eqref{eqz} given in section \ref{secref} below. As we pointed out in a remark following Theorem \ref{propref}, our result holds also in the higher-dimensional radial case \eqref{equr}, in the sense that for any $r_0>0$, we can construct a solution of equation \eqref{equr} such that its similarity variables version $w_{r_0}(y,s)$ behaves according to \eqref{cprofile0} with the parameters $d_i(s)$ given by \eqref{refequid1}.
\end{nb}
\begin{nb}
We are unable to say whether this solution has other characteristic points or not. In particular, we have been unable to find a solution with $\SS$ exactly equal to $\{0\}$. Nevertheless, let us remark that from the finite speed of propagation, we can prescribe more characteristic points, as follows:
\end{nb}
\begin{cor}[Prescribing more characteristic points]\label{cormore} Let $I=\{1,...,n_0\}$ or $I=\m N$ and for all $n\in I$, $x_n\in \m R$, $T_n>0$, $k_n \ge 2$ and $\zeta_{0,n}\in \m R$ such that
\begin{equation*}
x_n+T_n<x_{n+1}-T_{n+1}.
\end{equation*}
Then, there exists a blow-up solution $u(x,t)$ of equation \eqref{eq:nlw_u} in $\rm H^1_{\rm loc,u}\times \rm L^2_{\rm loc,u}(\m R)$ with $\{x_n\;|\; n\in I\} \subset \SS$, $T(x_n)=T_n$ and for all $n\in I$, 
\begin{equation*}
\left\|\vc{w_{x_n}(s)}{\ps w_{x_n}(s)} - \vc{\ds\sum_{i=1}^{k_n} (-1)^{i+1}\kappa(d_{i,n}(s))}0\right\|_{\H} \to 0\mbox{ as }s\to \infty,
\end{equation*} 
with
\begin{equation*}
\forall i=1,\dots,k_n,\;\;d_{i,n}(s) = -\tanh \zeta_{i,n}(s),\;\;
\zeta_{i,n}(s) = \bar \zeta_i(s) + \zeta_{0,n}
\end{equation*}
and $\bar \zeta_i(s)$ defined in \eqref{solpart}.
\end{cor}
\begin{nb}
Again, we are unable to construct a solution with $\SS = \{x_n\;|\; n\in I\}$.
\end{nb}

As one can see from \eqref{cprofile0} and \eqref{refequid1}, the solution we have just constructed in Theorem \ref{mainth} behaves like the sum of $k$ solitons as $s\to \infty$. In the literature, such a solution is called a {\it multi-soliton solution}. Constructing multi-soliton solutions is an important problem in nonlinear dispersive equations. It has already be done for the $L^2$ critical and subcritical nonlinear Schr\"odinger equation (NLS) (see Merle \cite{Mcmp90} and Martel and Merle \cite{MMihp06}), the $L^2$ critical and subcritical generalized Korteweg de Vries equation (gKdV) (see Martel \cite{Majm05}), and for the $L^2$ supercritical case both for (gKdV) and (NLS) equations in C\^ote, Martel and Merle \cite{CMMrmia11}.

\medskip

More generally, constructing a solution to some Partial Differential Equation with a prescribed behavior (not necessarily multi-solitons solutions) is an important question. We solved this question for (gKdV) in C\^ote \cite{Cjfa06,Cdmj07}, and also for parabolic equations exhibiting blow-up, like the semilinear heat equation with Merle in \cite{MZdmj97, MZcras96}, the complex Ginzburg-Landau equation in \cite{Zihp98} and with Masmoudi in \cite{MZjfa08}, or a gradient perturbed heat equation with Ebde in \cite{EZsema11}. In all these cases, the prescribed behavior shows a convergence to a limiting profile in some rescaled coordinates, as the time approaches the blow-up time.

\medskip

Surprisingly enough, in both the parabolic equations above and the supercritical dispersive equations treated in \cite{CMMrmia11}, the same topological argument is crucial to control the directions of instability. This will be the case again for the semilinear wave equation \eqref{eq:nlw_u} under consideration. Our strategy relies on two steps: 

\medskip

- Thanks to a dynamical system formulation, we show that controlling the similarity variables version $w(y,s)$ \eqref{eq:nlw_w} around the expected behavior \eqref{cprofile0} reduces to the control of the unstable directions, whose number is finite. This dynamical system formulation is essentially the same as the one that allowed us to show that all characteristic points are isolated in \cite{MZisol10}. Then, we solve the finite dimensional problem thanks to a topological argument based on index theory. This solves the problem without allowing us to prescribe the center of mass as required in \eqref{barycenter}. 

\medskip

- Performing a Lorentz transform on the solution we have just constructed, we are able to choose the center of mass as in \eqref{barycenter}.

\bigskip

Notice that in \cite{CMMrmia11, Mcmp90, Majm05}, the construction was done backward in time and used a compactness argument. We emphasize that this strategy is not available here, because we work in a light cone only; so we construct the desired solution forward in time directly via the topological argument.

\bigskip

This paper is organized in three sections: In Section \ref{secref}, we refine the blow-up behavior at a characteristic point 
and the geometry of the blow-up set 
and prove Theorem \ref{propref}
together with Corollary \ref{corref}. 
Then, in Section \ref{secweak}, we construct a multi-soliton solution in similarity variables. 
Finally, in Section \ref{seclorentz}, we translate this construction in the $u(x,t)$ formulation, and then use a Lorentz transform to prescribe the center of mass to finish the proof of Theorem \ref{mainth} and Corollary \ref{cormore}.

\section{Refined asymptotics near a characteristic point}\label{secref}
In this section, we prove Theorem \ref{propref} and Corollary \ref{corref}, 
refining the description given in \cite{MZajm11} for the blow-up behavior at a characteristic point
together with the geometry of the blow-up set. 
We proceed in two subsections, the first devoted to the proof of Theorem \ref{propref} and the second to the proof of Corollary \ref{corref}.

\subsection{Refined blow-up behavior near a characteristic point}

\begin{proof}[Proof of Theorem \ref{propref}] Consider $u(x,t)$ a blow-up solution of equation \eqref{eq:nlw_u} and $x_0\in \SS$. From the result of \cite{MZajm11} recalled in (iv) in page \pageref{old4}, we know that estimate \eqref{cprofile00} holds for some $k=k(x_0)\ge 2$ and $|\theta_1|=1$,
with continuous functions $d_i(s)= -\tanh \zeta_i(s)\in(-1,1)$ satisfying \eqref{equid00}. In order to conclude, we claim that it is enough to refine this estimate by showing that 
\begin{equation}\label{goal0}
\zeta_i(s) = \bar \zeta_i(s) + \zeta_0+o(1)\mbox{ as }s\to \infty,
\end{equation}
where $(\bar \zeta_i(s))_i$ \eqref{solpart} is the explicit solution to system \eqref{eq:tl}. Indeed, once this is proved, we can slightly modify the $\zeta_i(s)$ by setting $\zeta_i(s)$ exactly equal to $\bar \zeta_i(s) + \zeta_0$ (as required in \eqref{refequid}) and still have \eqref{cprofile00} hold, thanks to the following continuity result for the solitons $\kappa(d)$ \eqref{defkd}:
\begin{equation}\label{contkd}
\|\kappa(d_1)-\kappa(d_2)\|_{\q H_0} \le C|\arg\tanh d_1-\arg \tanh d_2|
\end{equation}
(see Lemma \ref{contk*} below for a more general statement). Thus, our goal in this section is to show \eqref{goal0}, where the $\zeta_i(s)=-\arg\tanh d_i(s)$ are the parameters shown in \eqref{cprofile00} proved in \cite{MZajm11}.

\medskip

From \cite[Proposition 3.2]{MZajm11}, we recall that $(\zeta_i(s))_{i=1,\dots,k}$ is in fact a $C^1$ function satisfying the following ODE system for $i=1,\dots,k$:
\begin{equation}\label{eqz}
\frac 1{c_1} \zeta_i' = e^{-\frac 2{p-1}(\zeta_i-\zeta_{i-1})}
-e^{-\frac 2{p-1}(\zeta_{i+1}-\zeta_i)}+R_i\mbox{ where }
R_i= O\left(\frac 1{s^{1+\eta}}\right)\mbox{ as }s\to \infty,
\end{equation}
for some explicit constant $c_1=c_1(p)>0$, and a fixed small constant $\eta=\eta(p)>0$, with the convention that $\zeta_0(s) \equiv -\infty$ and $\zeta_{k+1}(s) \equiv +\infty$. (Systems similar to \eqref{eqz} also appear in other contexts, for example, in the boundary layer formation for the real Ginzburg-Landau equation, see \cite{CaPeCPAM89,FuHaJDDE89}).

\medskip

We proceed in two parts: we first study system \eqref{eqz} without the rest term (i.e. when all $R_i\equiv 0$), then we take into account the full system and conclude the proof.

\bigskip

{\bf Part 1: The ODE system with no rest term}

Our system \eqref{eqz} with no rest terms is stated in \eqref{eq:tl}.

We proceed in 4 steps: We first give explicit solutions for system \eqref{eq:tl}. Then, we study its critical points and give a Lyapunov functional for it. In the third step, we find a compact in $\m R^k$ stable by the flow of system \eqref{eq:tl}. Finally, applying Lyapunov's theorem we show that any bounded solution is asymptotically close to one of the explicit solutions given in the first step.

\medskip

\emph{Step 1: Explicit solutions for system \eqref{eq:tl}}\label{step1} Let us introduce
\begin{equation}\label{defgi}
\ds \gamma_i = (p-1) \left(-i + \frac{k+1}{2} \right).
\end{equation}
We look for a solution of system \eqref{eq:tl} obeying the following ansatz
\begin{equation}\label{ansatz}
\zeta_i(s) = -\frac{\gamma_i}2\log s + \alpha_i, 
\end{equation}
we get the following necessary and sufficient condition: for all $i=2,\dots,k$,
\[
e^{-\frac{2}{p-1} (\alpha_i - \alpha_{i-1})} =  \frac{1}{2c_1} \ds\sum_{j=1}^{i-1} \gamma_j = -\frac{1}{2c_1} \ds\sum_{j=i}^k \gamma_j = \frac{(p-1)}{4c_1} (i-1)(k+1-i),
\]
which makes a one parameter family of solutions, for example characterized by its center of mass $\frac 1k\sum_{i=1}^k \zeta_i(s)$ (which in fact remains independent of time). Fixing the center of mass to be zero, we obtain the following particular solution
\begin{equation*}
\bar \zeta_i(s) = -\frac{\gamma_i}{2} \log s + \bar\alpha_i, 
\end{equation*}
already defined in \eqref{solpart}, where $\alpha_i=\bar \alpha_i(p,k)$ are uniquely determined by 
\begin{equation}\label{defalphai}
\displaystyle\sum_{i=1}^k \bar \alpha_i =0,\;\;e^{-\frac{2}{p-1} (\bar \alpha_i - \bar \alpha_{i-1})}= \frac{(p-1)}{4c_1} (i-1)(k+1-i),\;i= 2,\dots,k.
\end{equation}
In particular, all the other solutions obeying the ansatz \eqref{ansatz} are obtained as
\begin{equation}\label{onep}
\zeta_i(s) = \bar \zeta_i(s) + \zeta_0
\end{equation}
where $\zeta_0$ is the constant value of the center of mass $\frac 1k\sum_{i=1}^k \zeta_i(s)$.
Let us remark that
\begin{equation}\label{antis}
\forall s>0,\;\;\bar \zeta_i(s) = -\bar \zeta_{k-i}(s).
\end{equation}
Indeed, from the definition \eqref{defgi} of $\gamma_i$ and system \eqref{eq:tl}, we see that $(-\bar \zeta_{k-i}(s))_i$ is also a solution of system \eqref{eq:tl} obeying the ansatz \eqref{ansatz}. Therefore, as in \eqref{onep}, we have for all $i=1,\dots,k$ and $s>0$, $-\bar \zeta_{k-i}(s)=\bar \zeta_i(s) +\bar \zeta_0$, where 
\[ \bar \zeta_0= \frac 1k \sum_{i=1}^k (-\bar \zeta_{k-i}(s))= - \frac 1k\sum_{j=1}^k \bar \zeta_j(s)=0, \]
and \eqref{antis} follows. 

\medskip
 
\emph{Step 2: Critical points and a Lyapunov functional for system \eqref{eq:tl}}

We now look at a perturbation $\zeta(s) = (\zeta_i(s))_{i=1,\dots,k}$ of this solution. Denote
\begin{equation}\label{defxi0}
 \xi_i(\tau) = \frac 2{p-1}(\zeta_i(s) - \bar \zeta_i(s))\mbox{ where }\tau =\log s
 \end{equation}
and assume that the maximal solution exists on some interval $[0,\tau_\infty)$ where either $\tau_\infty$ is finite or $\tau_\infty=\infty$.  
We assume that
\begin{equation}\label{zerobarycenter} 
\sum_{i=1}^k \xi_i(0) =0
\end{equation}
Then, $\sum_{i=1}^k \xi_i(\tau) =0$ for all $\tau \in [0,\tau_\infty)$ and the $\xi_i$ satisfy the system 
\begin{equation} \label{eq:ptl}
\begin{cases}
\displaystyle \dot \xi_1 = -\sigma_1 (e^{-  (\xi_{2} - \xi_1) }-1), \\
\displaystyle \dot \xi_i = \vphantom{\int_f^f} \sigma_{i-1} (e^{ -  (\xi_i - \xi_{i-1}) } -1) -\sigma_i (e^{-  (\xi_{i+1} - \xi_i) }-1),
 & i=2, \dots, k-1 \\
\displaystyle \dot \xi_k = \sigma_{k-1} (e^{ - (\xi_{k} - \xi_{k-1}) } -1).
\end{cases}
\end{equation}
where 
\begin{equation}\label{defsi}
\sigma_i=\frac{i(k-i)}2.
\end{equation}
Denote 
\begin{equation}\label{defbi}
\displaystyle b_i(\tau) = \sigma_{i-1}(e^{ -  (\xi_i(\tau) - \xi_{i-1}(\tau)) } -1)\mbox{ for }i=2, \dots,k-1,\;\;b_1=b_{k+1}=0,
\end{equation}
 so that 
\begin{equation}\label{xbb}
\forall i=1,\dots,k,\;\;\dot \xi_i = b_i-b_{i+1}
\end{equation}
 and consider
\begin{equation} \label{def:bB}
b(\tau) = \min \{ b_i(\tau) | i = 1, \dots, k+1 \}, \quad B(\tau) = \max \{ b_i(\tau) | i = 1, \dots, k+1 \}.
\end{equation}
Note from \eqref{defbi} that
\begin{equation}\label{minbB}
b(\tau)\le 0\le B(\tau).
\end{equation}
\begin{prop}[The critical point and Lyapunov functionals of system \eqref{eq:ptl} under the condition \eqref{zerobarycenter}] \label{lyapunov}
The only critical point of 
system \eqref{eq:ptl} under the condition \eqref{zerobarycenter} 
is $\xi_i=0$.
Moreover, the functions
$B$ and $-b$ are Lyapunov functionals for 
system \eqref{eq:ptl}. In addition, $B-b$ is (strictly) decreasing, except if $\xi_1(\tau) \equiv \cdots \equiv \xi_k(\tau) \equiv 0$.
\end{prop}
\begin{proof}
Regarding the critical points: note that $\dot \xi_1 =0$ if and only if $\xi_2=\xi_1$. By a straightforward induction one sees that all $\xi_i$ are equal. As their sum is 0, the only critical point is $\xi_1 = \dots =\xi_k =0$.

\medskip

Let us now prove that $B$ is nonincreasing
along the flow; the argument for $-b$ will be similar.
Hence, let $\boldsymbol \xi(\tau) =(\xi_1(\tau),\dots,\xi_k(\tau))$ be a solution of \eqref{eq:ptl} such that
\begin{equation}\label{outside}
\boldsymbol \xi(\tau)\neq 0\mbox{ for any }\tau\mbox{ in the domain of definition}.
\end{equation}
Define 
\[
J(\tau) = \{ i \in \llbracket 1, k+1 \rrbracket\; |\; b_i(\tau) = B(\tau) \},
\]
 the set of indices $i$ for which $b_i$ is maximum at time $\tau$. The following lemma allows us to conclude:
\begin{lem} \label{raphael}For all $\tau_0\in [0, \tau_\infty)$, there exist $\e = \e(\tau_0)>0$ such that for all $i \in J(\tau_0)\cap\llbracket 2,k\rrbracket$, $b_i(t) < B(\tau_0)$ for all $t \in (\tau_0,\tau_0+\e)$.
\end{lem}
Indeed, assuming this lemma, let us show that for all $i=1,\dots,k+1$, there exists $\epsilon_i>0$ such that
\begin{equation}\label{goali}
\forall t\in(\tau_0, \tau_0+\epsilon_i),\;\;b_i(\tau)\le B(\tau_0).
\end{equation}
If $i=1,$ or $i=k+1$, then \eqref{goali} is obvious from \eqref{defbi} and \eqref{minbB}.\\
If $2\le i\le k$ and $i\in J(\tau_0)$, then \eqref{goali} is clear from Lemma \ref{raphael}.\\
If $2\le i\le k$ and $i\not\in J(\tau_0)$, then $b_i(\tau_0)<B(\tau_0)$ by definition of $J(\tau_0)$ and \eqref{goali} follows by continuity of $b_i(\tau)$.\\ 
By connectedness, it follows that $B$ is nonincreasing on the whole interval of definition of the solution.
The argument for $-b$ is quite similar.\\
In particular $B(\tau)-b(\tau)$ is nonincreasing too. Let us show that it is in fact decreasing. 
Since \eqref{outside} holds, it follows that either $B(\tau_0)>0$ or $-b(\tau_0)>0$ (otherwise $B(\tau_0)=b(\tau_0)=0$ by \eqref{minbB}, hence $b_i(\tau_0)=0$ and $\xi_i(\tau_0)=\xi_1(\tau_0)=0$ by \eqref{def:bB}, \eqref{defbi} and \eqref{zerobarycenter}, which is a contradiction by \eqref{outside}). Hence, using a similar argument to the proof of \eqref{goali}, we see that either $B$ or $-b$ is decreasing. Thus, $B-b$ is decreasing, which is the desired conclusion for Proposition \ref{lyapunov}. It remains to prove Lemma \ref{raphael} in order to finish the proof of Proposition \ref{lyapunov}.

\begin{proof}[Proof of Lemma \ref{raphael}]
Let $\llbracket \aa ,\bb \rrbracket \subset J(\tau_0)$ be a maximal interval of integers included in $J(\tau_0)$. As $J(\tau_0)$ is a union of such intervals, it is enough to prove Lemma \ref{raphael} for all $i\in \llbracket \aa ,\bb \rrbracket \cap \llbracket 2 ,k \rrbracket$. \\
Notice that for $i=2,\dots,k$, $b_i$ has the sign of $(\xi_{i-1} - \xi_i)$, and that 
\begin{equation} \label{diffb}
\dot b_i = \sigma_{i-1}  e^{ - (\xi_i - \xi_{i-1}) } (\dot \xi_{i-1} - \dot \xi_{i})
\end{equation}
 has the sign of $(\dot \xi_{i-1} - \dot \xi_i)$.
Now, using \eqref{xbb}, we write
\begin{equation} \label{diffx}
\dot \xi_{i-1} - \dot \xi_i = b_{i-1} -2b_i + b_{i+1}.
\end{equation}

\noindent\emph{Case 1: $\llbracket \aa ,\bb \rrbracket \subset \llbracket 2 ,k \rrbracket$}.

In particular, as $\aa-1 \notin J(\tau_0)$ and $\bb+1 \notin J(\tau_0)$, we get $b_{\aa-1}(\tau_0) < B(\tau_0)$ and $b_{\bb+1}(\tau_0) < B(\tau_0)$, and this shows that
\begin{equation} \label{p1}
\begin{cases}
\dot b_\aa(\tau_0) <0, \quad \dot b_\bb(\tau_0) <0, \\
\dot b_i(\tau_0) = 0, \quad \text{for all } i \text{ such that } \aa < i < \bb.
\end{cases}
\end{equation}
If $i=\aa$ or $i=\bb$, as $b_i(\tau_0)$ is maximum, we see that $\dot b_i (\tau_0)<0$, so that Lemma \ref{raphael} holds for this $i$.

\medskip

Now assume $\aa < i <\bb $, then $b_{i-1} (\tau_0)= b_i(\tau_0) = b_{i+1}(\tau_0) = B(\tau_0)$, so that $\dot \xi_{i-1}(\tau_0) - \dot \xi_i(\tau_0) =0$ and then $\dot b_i(\tau_0)=0$ from \eqref{diffb} and \eqref{diffx}. 
We will show in fact that a higher derivative of $b_i$ is negative at $\tau=\tau_0$ which will conclude the proof of Lemma \ref{raphael} for this $i$. 
More precisely, we prove the following:

\medskip

\noindent\emph{Claim.} 
Let $d(i) = \min \{ i-\aa, \bb-i \}$. If $d(i) \ge 1$, then
\begin{equation}
\dot b_i(\tau_0) = \dots = b_i^{(d(i))}(\tau_0) =0\mbox{ and } b_i^{(d(i)+1)}(\tau_0) < 0.
\end{equation}

To prove the Claim, for $n \in \llbracket 1, \lfloor (\bb-\aa)/2 \rfloor \rrbracket$ where $\lfloor z\rfloor$ stands for the integer part of $z\in\m R$, consider the proposition
\[ 
P_n: \quad b_{\aa-1+n}^{(n)}(\tau_0) <0, \quad  b_{\bb+1-n}^{(n)}(\tau_0) <0, \quad \text{and } \ \forall i \in \llbracket \aa+n, \bb-n \rrbracket, \ b_i^{(n)} (\tau_0) =0.
\]
In some sense, this proposition relies on an inductive mechanism, where a negative higher derivative propagates from $i=\bb$ in the left direction, affecting after each step the next derivative of the left neighbor. A similar phenomenon starts from $i=\aa$ and goes to the right.
We will prove proposition $P_n$ by induction on $n$.\\
Note that \eqref{p1} proves $P_1$. Let $n \ge 2$ and assume that $P_1, \dots P_{n-1}$ hold. In particular, $P_{n-1}$ gives
\[  
b_{\aa-2+n}^{(n-1)}(\tau_0) <0, \quad  b_{\bb+2-n}^{(n-1)}(\tau_0) <0, \text{ and } \forall i \in \llbracket \aa-1+n; \bb+1-n \rrbracket, \ b_i^{(n-1)} (\tau_0) =0. 
\]
Differentiating \eqref{diffx} $(n-1)$ times gives
\[ \xi_{i-1}^{(n)} - \xi_i^{(n)} = b_{i-1}^{(n-1)} -2b_i^{(n-1)} + b_{i+1}^{(n-1)}. \]
From the previous two statements, we see that
\begin{gather*}
\xi_{\aa-2+n}^{(n)}(\tau_0) - \xi_{\aa-1+n}^{(n)}(\tau_0) <0, \quad \xi_{\bb-n}^{(n)}(\tau_0) - \xi_{\bb+1-n}^{(n)}(\tau_0) <0, \\
\text{and for } i \in \llbracket \aa+n, \bb-n \rrbracket, \quad \xi_{i-1}^{(n)}(\tau_0) - \xi_{i}^{(n)}(\tau_0) =0.
\end{gather*}
Propositions $P_1, \dots P_{n-1}$ show (in the same way) that for $i \in \llbracket \aa-1+n; \bb+1-n \rrbracket$
\begin{equation} \label{diffxn}
\dot \xi_{i-1}(\tau_0) - \dot \xi_{i}(\tau_0) = \cdots = \xi_{i-1}^{(n-1)}(\tau_0) - \xi_{i}^{(n-1)}(\tau_0) =0.
\end{equation}
Now differentiate \eqref{diffb} $(n-1)$ times using the Leibniz and Faà di Bruno formulas, we see that at $\tau_0$, the only term remaining is the one with $n$ derivative on $\xi_{i-1} - \xi_i$, i.e.
\[ 
b_i^{(n)}(\tau_0) = \frac{(i-1)(k+1-i)}{2\tau_0 } e^{ - \frac{2}{p-1} (\xi_i(\tau_0) - \xi_{i-1}(\tau_0)) } (\xi_{i-1}^{(n)}(\tau_0) - \xi_{i}^{(n)}(\tau_0)). 
\]
Hence, we then deduce that
\[ 
b_{\aa-1+n}^{(n)}(\tau_0) < 0, \quad b_{\bb+1-n}^{(n)}(\tau_0) <0, \quad \text{and for } i \in \llbracket \aa+n, \bb-n \rrbracket, \ b_i^{(n)}(\tau_0) =0. 
\]
This is $P_n$, which concludes the induction. Fixing $i$ and $d \in \llbracket 1, \dots, d(i) \rrbracket$, we see that $P_d$ gives $b_i^{(d)}(\tau_0) =0$. $P_{d(i)+1}$ gives $b_i^{(d(i)+1)}(\tau_0) <0$. Hence the Claim is proved.

\medskip

From the Claim (and \eqref{p1} in the case $d(i) =0$) and Taylor's expansion, we see that $b_i(t) - b_i(\tau_0) \sim b_i^{(d(i)+1)}(\tau_0) (t-\tau_0)^{d(i)+1}$. In particular, as $b_i^{(d(i)+1)}(\tau_0) <0$, for some small enough $\e>0$, we see that for $t \in (\tau_0,\tau_0+\e)$, $b_i(t) < b_i(\tau_0) = B(\tau_0)$.
This concludes the proof of Lemma \ref{raphael} in the case where $\llbracket \aa ,\bb \rrbracket \subset \llbracket 2 ,k \rrbracket$. 

\medskip

\noindent\emph{Case 2: $\aa=1$ or $\bb=k+1$}.

 We only treat the case where $\aa=1$, the other case being similar. Moreover, we only sketch the proof, since it uses the same techniques as Case 1 above.\\
 Note first that since $1\in J(\tau_0)$ and $b_1(\tau_0)=0$ by \eqref{defbi}, it follows that $B(\tau_0)=0$.\\
Then, we claim that 
\begin{equation}\label{bbbas}
\bb\le k-1.
\end{equation}
 Indeed, if $\bb=k$, then recalling that $b_{k+1}(\tau_0)=0$ by \eqref{defbi}, we see that $J(\tau_0)=\llbracket 1,k+1\rrbracket$ and $\llbracket \aa, \bb\rrbracket=\llbracket 1,k\rrbracket$ is not maximal in $J(\tau_0)$, which is a contradiction.\\
If $\bb=k+1$, then for all $i=2,\dots,k$, $b_i(\tau_0)=0$ and $\xi_i(\tau_0)=\xi_1(\tau_0)=0$ by \eqref{def:bB}, \eqref{defbi} and \eqref{zerobarycenter}, which is a contradiction by \eqref{outside}. Thus, \eqref{bbbas} holds.

\medskip

If $\bb=1$, then 
$ \llbracket \aa ,\bb \rrbracket \cap \llbracket 2 ,k \rrbracket=\emptyset$ and we have nothing to prove.

If $\bb\ge 2$ (which means that $k\ge 3$ by \eqref{bbbas}), since $\bb+1\not \in J(\tau_0)$, arguing as for \eqref{p1}, we see that 
\begin{equation}\label{bp1}
\dot b_\bb(\tau_0) <0,
\end{equation}
and the conclusion of Lemma follows for $i=\bb$. More generally, 
as in Case 1, the following claim allows us to conclude:

\medskip

\noindent\emph{Claim.}
For all $i=2,\dots,\bb$, we have 
\begin{equation}
\dot b_i(\tau_0) = \dots = b_i^{(d(i))}(\tau_0) =0\mbox{ and } b_i^{(d(i)+1)}(\tau_0) < 0\mbox{ where }d(i) =\bb-i.
\end{equation}

The proof of this claim uses the same iterative procedure as in Case 1, based on an induction for the following property for all $n=1,\dots,\bb+2$:
\[ 
\overline P_n: \quad  b_{\bb+1-n}^{(n)}(\tau_0) <0, \quad \text{and } \ \forall i \in \llbracket 2, \bb-n \rrbracket, \ b_i^{(n)} (\tau_0) =0.
\]
Note that $\overline P_1$ follows by \eqref{bp1}.\\
Note also that unlike the property $P_n$ is Case 1 where the negative higher derivative is propagating both from the right and from the left, here, it propagates only from the right. The non propagation from the left is replaced by the information that that $b_1(\tau)$ is identically zero, hence $b_1^{(j)}(\tau_0)=0$ for all $j\in \m N$.\\
For more details, see Case 1. This concludes the proof of Lemma \ref{raphael}.
\end{proof}

As noticed above, Proposition \ref{lyapunov} then proceeds from Lemma \ref{raphael}
\end{proof}

\emph{Step 3: A compact stable by the flow of \eqref{eq:ptl} under condition \eqref{zerobarycenter}}

From the definition \eqref{defbi} of $b_i(\tau)$ and the equations \eqref{diffb} and \eqref{diffx}, we write for all $i=2,\dots,k$ and $\tau \in [0, \tau_\infty)$, 
\begin{equation} \label{edo:bi}
\dot b_i = (b_i + \sigma_{i-1}) (b_{i-1} - 2b_i + b_{i+1}) \quad i = 2, \dots, k,
\end{equation}
where we set by convention 
\begin{equation}\label{b1bk+1}
b_1(\tau)\equiv b_{k+1}(\tau)\equiv 0.
\end{equation} 
Note that thanks to condition \eqref{zerobarycenter} and under the condition
\begin{equation}\label{above-sigma-i}
b_i(\tau)>-\sigma_{i-1},
\end{equation}
 this system is equivalent to system \eqref{eq:ptl}. 
We claim the following:
\begin{prop}[Compacts stable by the flow of system \eqref{edo:bi}]\label{propcpt} For all $\eta\in(0, \frac 15]$ and $A\ge 0$, the compact $\prod_{i=2}^k [-\sigma_{i-1}+\eta, A]$ is stable by the flow of system \eqref{edo:bi}. In particular, any solution of system \eqref{edo:bi} whose initial data is in that compact is global. 
\end{prop}
\begin{nb}
From the equivalence between system \eqref{edo:bi} and system \eqref{eq:ptl} under conditions \eqref{zerobarycenter} and \eqref{above-sigma-i}, any solution to the Cauchy problem for system \eqref{eq:ptl} under the condition \eqref{zerobarycenter} exists globally in time. The same holds for any solution to system \eqref{eq:tl} too. 
\end{nb}
\begin{proof} Let $\eta\in(0, \frac 15]$ and $A\ge 0$, and consider initial data for \eqref{edo:bi} such that
\[
\forall i=2,\dots,k,\;\;-\sigma_{i-1}+\eta \le b_i(0)\le A.
\]
In particular, we have $B(0)\le A$ where $B(\tau)$ is defined in \eqref{def:bB}. Since $B(\tau)$ is nonincreasing by Proposition \ref{lyapunov}, it follows that
\[
\forall i=2,\dots,k,\;\;b_i(\tau)\le B(\tau) \le B(0)\le A.
\]
Now we argue by contradiction and assume that for some $\tau>0$, there exists $i=2,\dots,k$ such that $b_i(\tau)<-\sigma_i+\eta$. Choosing the lowest $\tau$, we end up by continuity with some $\tau^*\ge 0$ such that 
\begin{eqnarray}
(i)&&\forall \tau\in[0, \tau^*],\;\;\forall i=1,\dots,k,\;\; b_i(\tau) \ge -\sigma_i+\eta,\label{above}\\
(ii)&&b_j(\tau^*)=-\sigma_j+\eta\mbox{ and }\dot b_j(\tau^*)\le 0\mbox{ for some }j=2,\dots,k,\label{abovej}
\end{eqnarray} 
on the one hand. On the other hand, $(\sigma_i)_i$ is a (strictly) convex family of semi-integers, so that in particular, 
\[
\forall i =1,\dots,k,\;\;-\sigma_{i-1} + 2\sigma_i - \sigma_{i+1} \ge 1/2
\]
and remarking that 
\[
b_{j-1}(\tau^*)\ge - \sigma_{j-1}\mbox{ and }b_{j+1}(\tau^*)\ge - \sigma_{j+1}
\]
(use \eqref{above} if the indices are between $2$ and $k$, and use \eqref{b1bk+1} and the definition \eqref{defsi} of $\sigma_i$ if the indices are either $1$ or $k+1$),  we see from \eqref{edo:bi} and \eqref{abovej} that 
\begin{align*}
\dot b_j(\tau^*) & =  (b_j(\tau^*) + \sigma_j) (b_{j-1}(\tau^*) - 2b_j(\tau^*) + b_{j+1}(\tau^*)) \\
& \ge  \eta ( - \sigma_{j-1} + 2 \sigma_j- 2 \eta - \sigma_{j+1}) \ge \e (1/2 - 2\eta) > 0
\end{align*}
since we choose $\eta \in(0, \frac 15]$. This is a contradiction with \eqref{abovej}. 
\end{proof}

\emph{Step 4: Asymptotic behavior of solutions to \eqref{eq:ptl} under condition \eqref{zerobarycenter}}

We endow $\m R^k$ with the $\ell^2$ norm. We show in the following that any solution to system \eqref{eq:tl} approaches the particular family of solutions given in Step 1:
\begin{prop}[Asymptotic behavior for system \eqref{eq:ptl} under condition \eqref{zerobarycenter}] \label{prop:tl_conv}$ $\\
(i) Let $(\xi_i(\tau))_{i=1,\dots,k}$ be a solution to system \eqref{eq:ptl} under condition \eqref{zerobarycenter}, with initial data at $\tau=0$ satisfying
\begin{equation}\label{compact0}
\forall i=1,\dots,k,\;\;|\xi_i(0)|\le C_0
\end{equation}
for some $C_0>0$. Then, the solution is defined for all $\tau\ge 0$ and there exists $C_1(C_0)>0$ such that
\[ 
\forall \tau \ge 0,\;\;\sup_{i}|\xi_i(\tau)|\le C_1 e^{-\tau}.
\]
(ii)  Let $(\zeta_i(s))_{i=1,\dots,k}$ be a solution to system \eqref{eq:tl} with initial data given at $s=1$. Then, the solution is defined for all $s\ge 1$ and 
\[
\forall s\ge 1,\;\; \sup_{i} | \zeta_i(s) - (\bar \zeta_i(s)+\zeta_0) | \le C s^{-1}\mbox{ with } \zeta_0 = \frac{1}{k} \sum_{k=1}^n \zeta_i(1),
\]
where the $(\bar \zeta_i(s))$ is the explicit solution of system \eqref{eq:tl} introduced in Step 1 above. 
\end{prop}
\begin{proof}
Let us first derive (ii) from (i), then we will prove (i).\\
(ii) Consider $(\zeta_i(s))_{i=1,\dots,k}$ a solution to system \eqref{eq:tl} with initial data given at $s=1$. Since the center of mass is conserved in time, introducing
\[
\xi_i(\tau) =  \frac 2{p-1}\left[\zeta_i(s) - \left(\bar \zeta_i(s)+\zeta_0\right)\right]\mbox{ where }\tau = \log s,
\] 
we see from Step 1 above that 
\begin{equation}\label{cond}
\forall \tau\ge 0,\;\;\sum_i \xi_i(\tau) =\sum_i \xi_i(0) =0
\end{equation}
 and $(\xi_i(\tau))_i$ satisfies \eqref{eq:ptl}. Thus, (ii) follows from (i).\\
(i) From the remark following Proposition \ref{propcpt}, we know that
$(\xi_i(s))_{i=1,\dots,k}$
is globally defined in time. 
\\
Introducing $b_i(\tau)$ as in \eqref{defbi}, we recall from Step 3 above the equivalence between system \eqref{edo:bi} and system \eqref{eq:ptl} under conditions \eqref{zerobarycenter} and \eqref{above-sigma-i}. In particular, using Proposition \ref{lyapunov}, we see that $(b_2(\tau),\dots, b_k(\tau))\equiv (0,\dots,0)$ is the only critical point of system \eqref{edo:bi} and that the functional $B-b$ is a Lyapunov functional, strictly decreasing except at the critical point (see Step 3 above). Since Proposition \ref{propcpt} provides us with a compact 
\[
K(C_0)=\prod_{i=2}^k [-\sigma_{i-1}+\eta, A]\text{ for some }\eta=\eta(C_0)>0\text{ and }A=A(C_0) 
\]
stable under the flow of system \eqref{edo:bi}, we see that Lyapunov's theorem applies to this system and yields the fact that $b_i(\tau) \to 0$ as $\tau \to 0$ (see Appendix \ref{applyap} below for the statement and the proof of the version of Lyapunov's theorem we use). From the relation \eqref{defbi} between $\xi_i$ and $b_i$ together withe zero barycenter condition \eqref{zerobarycenter}, we see that  
\begin{eqnarray*}
(i)&& \forall \tau \ge 0,\;\;|\xi_i(\tau)|\le C_2=C_2(C_0),\\
(ii)&& \xi_i(\tau) \to 0\mbox{ as }\tau \to +\infty.
\end{eqnarray*}
%
Since initial data are chosen in a compact (see \eqref{compact0} above), using the continuity with respect to initial data, for solutions of ODEs on a given time interval, we see that this convergence is uniform, in the sense that
\begin{equation}\label{uniform-xi}
\forall \epsilon>0,\ \exists \tau^*(C_0,\epsilon)>0 \quad \text{such that} \quad \forall \tau \ge \tau^*,\ \|\boldsymbol \xi(\tau)\|\le \epsilon.
\end{equation}
Linearizing system \eqref{eq:ptl} near the zero solution, we write
\begin{equation}\label{eqlin} 
\forall \tau \ge 0,\quad \left\| \dot {\boldsymbol \xi}(\tau) - M \boldsymbol \xi(\tau)\right\| \le C_3 \|\boldsymbol\xi(\tau) \|^2 \quad \text{for some} \quad C_3=C_3(C_0)>0,
\end{equation}
where 
the $k\times k$ matrix $M = (m_{i,j})_{(i,j) \in \llbracket 1,k \rrbracket}$, with 
\begin{equation}\label{defM}
m_{i,i-1}=\sigma_{i-1},\quad m_{i,i} = - (\sigma_{i-1} + \sigma_i),\quad m_{i,i+1}=\sigma_i, m_{i,j}=0 \text{ if }|i-j|\ge 2
\end{equation}
and $\sigma_i$ is defined in \eqref{defsi}. 
We claim the following:
\begin{lem}[Eigenvalues of $M$]\label{eigenm} 
The matrix $M$ is diagonalizable, with real eigenvalues 
\begin{equation}\label{defmi}
-m_i \equiv - \frac{i(i-1)}2,\mbox{ for } i=1,...,k,
\end{equation}
and the associated eigenvectors $\boldsymbol{e}_i$ normalized for the $\ell^\infty$ norm.
 If $i=1$, then $e_1 = {}^t(1,\dots,1)$.
\end{lem}
\begin{nb}
\end{nb}
\begin{proof}
Since $M$ is symmetric, it is diagonalizable with real eigenvalues. Furthermore, we can compute
\begin{equation*}
 (M \boldsymbol \xi, \boldsymbol \xi) = - \sum_{i=1}^{k-1} \sigma_i (\xi_{i+1} - \xi_{i})^2. 
\end{equation*}
In particular, $M(x,x) =0$ if and only if $(\boldsymbol \xi, {}^t(1,\dots, 1))$ is linearly dependent, so that $M$ has $0$ as an eigenvalue with eigenvector ${}^t(1,\dots, 1)$, and the other eigenvalues are negative.\\
 The proof of the exact value \eqref{defmi} of the eigenvalues relies on clever transformation of the matrix $M$, which are somehow long. We leave them to Appendix \ref{appmi}. See Appendix \ref{appmi} for the end of the proof of Lemma \ref{eigenm}.
\end{proof}

With this lemma, we carry on the proof of Proposition \ref{prop:tl_conv}.
Now, as $\sum_i \xi_i(\tau) =0$, we have from Lemma \ref{eigenm} that
\[ 
(M \boldsymbol \xi, \boldsymbol \xi) \le  - \| \boldsymbol \xi \|^2, 
\]
so that
\[ \frac{d}{d\tau} \| \boldsymbol \xi(\tau) \|^2 \le 
- 2 \| \boldsymbol \xi(\tau) \|^2 + C_4 \| \boldsymbol \xi(\tau) \|^3
\]
for some $C_4=C_4(C_0)>0$.
Since $\boldsymbol \xi(\tau)\to 0$ as $\tau \to \infty$, uniformly with respect to initial data satisfying \eqref{compact0} (see \eqref{uniform-xi} above), we see that  $\| \boldsymbol \xi(\tau) \| \le C_1e^{- \tau}=C_1s^{-1}$
for some $C_1=C_1(C_0)$. This concludes the proof of Proposition \ref{prop:tl_conv}.
\end{proof}

\bigskip

{\bf Part 2: Proof for the perturbed ODE}

We now turn to the equation \eqref{eqz} satisfied by $(\zeta_i(s))_i$, which is a perturbation of the autonomous system \eqref{eq:tl} studied in Part 1. We will prove in fact that when $s\to \infty$, $(\zeta_i(s))_i$ approaches one of the particular solutions of the autonomous system \eqref{eq:tl} introduced in Step 1 of Part 1.
More precisely, we will  prove a more accurate version of \eqref{goal0}, by showing the existence of $\zeta_0 \in \m R$ such that 
\begin{equation}\label{hadaf}
\forall i \in \llbracket 1,k \rrbracket,\; \zeta_i(s) = \bar \zeta_i(s) + \zeta_0 + O(s^{-\eta})\mbox{ as }s\to \infty,
\end{equation}
where $(\bar \zeta_i(s))_i$ is introduced in \eqref{solpart}.

Recall that we already have from \eqref{equid00} a less accurate estimate, namely that
\begin{equation*}
\forall i \in \llbracket 1,k \rrbracket,\; \zeta_i(s) = \bar \zeta_i(s) + O(1)\mbox{ as }s\to \infty.
\end{equation*}
We write from system \eqref{eqz} that 
\begin{equation}\label{limitz}
\sum_{i=1}^k \dot \zeta_i(s) = O\left(\frac 1{s^{1+\eta}}\right),\mbox{ hence } 
\frac 1k \sum_{i=1}^k  \zeta_i(s) =l+ O\left(\frac 1{s^\eta}\right)\mbox{ as }s\to \infty
\end{equation}
for some $l\in\m R$. Introducing 
\begin{equation}\label{defxi}
\xi_i(\tau) =\frac 2{p-1}\left[ \zeta_i(s) -  (\bar \zeta_i(s) + \frac 1k \sum_{i=1}^k  \zeta_i(s) )\right]
\quad \text{with} \quad \tau = \log s,
\end{equation}
we see from \eqref{limitz} and the definition \eqref{solpart} of $\bar \zeta_i(s)$ that
$\boldsymbol \xi = {}^t(\xi_1, \dots, \xi_k)$ satisfies
\begin{equation}\label{eqxi}
\forall \tau \ge \tau_0,\;\;
\left\|\dot{\boldsymbol \xi}(\tau)- \tilde f(\boldsymbol \xi(\tau))\right\| \le C_0e^{-\eta\tau},\;\;
\frac 1k\sum_i \xi_i(\tau)=0,\;\;
\left\|\boldsymbol \xi(\tau)\right\|\le C_0,
\end{equation}
for some positive $C_0$ and $\tau_0$,
where $\tilde f$ is the autonomous nonlinearity in the right-hand side of system \eqref{eq:ptl}.\\
In particular, as we will show below, $(\xi_i(\tau))_{i=1,\dots,k}$ will be close to some solution of system \eqref{eq:ptl} for $\tau$ large enough. Since solutions to \eqref{eq:ptl} converge uniformly to $0$ by Proposition \ref{prop:tl_conv}, $(\xi_i(\tau))$ will be as close to $0$ as we wish, provided that $\tau$ is large enough. More precisely, we claim the following:

\medskip
\emph{Claim.} For any $\e>0$, there exists $\hat \tau = \hat \tau(\epsilon)>0$ such that 
\begin{equation}\label{th}
\forall i=1,\dots,k,\;\;|\xi_i(\hat \tau)|\le \e.
\end{equation}
\noindent Let us use first this claim to finish the proof, then we will prove it. 

From the analysis carried out for the autonomous system \eqref{eq:ptl} in the proof of Proposition \ref{prop:tl_conv}, we linearize system \eqref{eqxi} then use the spectral properties of the matrix $M$ \eqref{defM} to write
\[
\forall \tau \ge \hat \tau,\;\;\frac{d}{d\tau} \| \boldsymbol \xi(\tau) \|^2 \le - 2  \| \boldsymbol \xi(\tau) \|^2 + C \| \boldsymbol \xi(\tau) \|^3+ Ce^{-\eta\tau}.
\]
Taking $\e$ small enough and starting from the estimate \eqref{th} at $\tau=\hat \tau$, we get by a classical integration
\[
\forall \tau\ge \hat \tau,\;\;\| \boldsymbol \xi(\tau) \| \le Ce^{-\eta \tau}
\]
(we recall here that already in \cite{MZajm11}, the constant $\eta=\eta(p)>0$ was chosen small enough). Using the definition \eqref{defxi} of $\xi_i(s)$ together with \eqref{limitz}, we see that \eqref{hadaf} holds and so does the conclusion of Theorem \ref{propref} too.
 It remains then to prove the Claim in order to conclude the proof of Theorem \ref{propref}.\\ 
For any $\bar \tau\ge \tau_0$, let us introduce $(\bar \xi(\tau))_{\bar \tau, i=1,\dots,k}$ the solution of the unperturbed system \eqref{eq:ptl} with initial data 
\begin{equation}\label{initial}
\bar \xi_{\bar \tau,i}(0)= \xi_i(\bar \tau).
\end{equation}
Using the continuity of solutions to ODEs with respect to the coefficients of the equations, we write from \eqref{initial} and \eqref{eqxi} for any $L>0$, 
\begin{equation}\label{proche}
\sup_{i=1,\dots,k; \tau\in[\bar \tau, \bar \tau+L]}|\xi_i(\tau)- \bar \xi_{\bar \tau,i}(\tau-\bar \tau)|\le C(L) e^{-\eta\bar \tau}.
\end{equation}
Since we have from \eqref{eqxi},
\[
\forall i=1,\dots,k,\;\;|\bar \xi_{\bar \tau,i}(0)|\le C_0,\;\;
\sum_{j=1}^k \bar \xi_{\bar \tau,j}(0)=0,\mbox{ hence }\sum_{j=1}^k \xi_{\bar \tau, j}(\tau)=0\mbox{ for all }\tau \ge 0,
\]
given $\epsilon>0$, we see from Proposition \ref{prop:tl_conv} (i) that for some $\tau^*(C_0,\e)>0$, we have 
\[
\forall i=1,\dots,k,\;\;|\bar \xi_{\bar\tau,i}(\tau^*)|\le \frac \e 2.
\]
Using \eqref{proche} with $L=\tau^*(C_0,\epsilon)$,
we see that 
\[
\forall i=1,\dots,k,\;\;|\xi_i(\bar \tau+\tau^*)|\le |\bar \xi_{\bar \tau,i}(\tau^*)|+C(\tau^*)e^{-\eta\bar \tau}\le \frac \e 2 + C(\tau^*)e^{-\eta\bar \tau}\le \e
\]
provided that we take $\bar \tau=\hat \tau(C_0, \epsilon)$ large enough. Taking $\hat \tau = \bar \tau +\tau^*$, we see that the Claim is proved, and so is \eqref{hadaf}, \eqref{goal0} and Theorem \ref{propref} too, thanks to the reduction we wrote after giving \eqref{goal0}.
\end{proof}

\subsection{Refined geometrical estimates for the blow-up set}

This section is devoted to the proof of Corollary \ref{corref}, which consists in a refinement of estimate \eqref{chapeau0} itself coming from \cite{MZisol10}.

\begin{proof}[Proof of Corollary \ref{corref}] From translation invariance of equation \eqref{eq:nlw_u}, we may assume that $x_0=0$ and $T(x_0)=0$.
Up to replacing $u$ by $-u$, we know from Theorem \ref{propref} that 
\begin{equation}\label{profilew0}
\left\|\vc{w_0(s)}{\ps w_0(s)} - \vc{\ds\sum_{i=1}^{k} (-1)^i\kappa(d_i(s))}0\right\|_{\H} \to 0\mbox{ as }s\to \infty,
\end{equation}
where $k=k(0)\ge 2$,
\begin{equation}\label{refequidw0}
d_i(s) = -\tanh \zeta_i(s),\;\;\zeta_i(s) = \bar \zeta_i(s) + \zeta_0
\end{equation}
for some $\zeta_0\in \m R$, and $(\bar \zeta_i(s))_i$ introduced above in \eqref{solpart} is the solution of system \eqref{eq:tl} with zero center of mass.
From symmetry invariance, we may treat the case $x<0$ first, then, at the end of the proof, we will give indications on how to recover the case $x>0$.\\
{\bf Case $x<0$}: All that we need to do is to review the proof of estimate \eqref{chapeau0} in \cite{MZisol10} and mechanically improve its estimates thanks to the new refined blow-up behavior we have just proved with Theorem \ref{propref}.\\
In \cite{MZisol10}, we prove the following estimate for $w_x$, where $x<0$ with $|x|$ small:
\begin{lem}\label{lemtrapped}For all $\epsilon>0$, there exists $\delta>0$ and $L>0$ such that for all $x\in(-\delta,0)$ and $L_k \ge L$, we have 
\[
\left\|\vc{w_x(s_k)}{\ps w_x(s_k)}+\vc{\kappa\left({\bar d}^*_1(s_k)\right)}{0}\right\|_{\H}+|\bar\lambda_1-1| \le \epsilon,
\]
where
\begin{equation}\label{defbd1}
{\bar d}^*_1(s_k)= \frac{{\bar d}_1(s_k)}{1+{\bar \nu}_1(s_k)},\;\;{\bar d}_1(s_k)=d_1(S_k),\;\; 
{\bar \nu}_i(s_k)=[b-(1-{\bar d}_i(s_k))]x e^{s_k},
\end{equation}
\begin{equation}\label{defsk}
s_k = |\log |x||+L_k,\;\;S_k=-\log[|x|(1-b)+e^{-s_\k}]
\end{equation}
and
\begin{equation}\label{deflambda}
\lambda_1=\frac{(1-{\bar d}_1^2)^{\frac 1{p-1}}}{[(1+\bar \nu_1)^2-{\bar d}_1^2]^{\frac 1{p-1}}}.
\end{equation}
\end{lem}
\begin{proof}For the proof, see Section 3 in \cite{MZisol10}, in particular the proof of Proposition 3.10 in that paper. Nevertheless, let us summarize in the following the 3 main arguments of the proof, and refer the interested reader to \cite{MZisol10} for more details:\\
- applying the similarity variables' transformation \eqref{def:w} twice, we first recover an estimate on $u(x,t)$, then on $w_x(s)$,
but only on the interval $y\in (y_1(x,s),1)$, for some $y_1(x,s)>-1$;\\
- using a very good understanding of the dynamics of equation \eqref{eq:nlw_w} near the sum of decoupled solitons (the same dynamical study that we use later in this paper for the proof of Theorem \ref{mainth}, see Appendix \ref{appdyn} below), we recover the same estimate on the whole interval $y\in(-1,1)$;\\
- the estimate we recover on $w_x(s)$ shows in fact that, like $w_0$, $w_x(s)$ is still a sum of $k$ decoupled solitons, though the solitons are no longer ``pure'' (i.e. given by $\kappa(d)$ defined in \eqref{defkd}), but generalized, given by the family $\kappa^*(d, \nu)$ defined in \eqref{defk*}. As time increases, this family starts to loose its members, starting from the right soliton (with index $i=k$) up to the second soliton (with index $i=2$) which is lost at time $s=s_k$ given above in \eqref{defsk}. Thanks to an energy argument, we show that this unique left soliton is a ``pure'' soliton, in other words given by $-\kappa({\bar d}_1^*)$ where ${\bar d}_1^*(s_k)$ is defined above in \eqref{defbd1}.
\end{proof}
This lemma shows that $w_x(s_k)$ is close to $-\kappa({\bar d}_1^*)$.
As a matter of fact, we have the following trapping result from Merle and Zaag \cite{MZjfa07} which asserts that $w_x$ will eventually converge to a nearby soliton, with a near parameter:
\begin{prop}[A trapping criterion for non-characteristic points] There exist $\epsilon_0>0$ and $C_0>0$ such that if for some $x_0\in\m R$, $s_0\ge -\log T(x_0)$, $\theta \in \{ \pm 1 \}$, $d\in(-1,1)$ and $\epsilon\in(0, \epsilon_0]$, we have 
\begin{equation*}
\left\|\vc{w_{x_0}(s_0)}{\partial_s w_{x_0}(s_0)}-\theta\vc{\kappa(d)}{0}\right\|_{\H}\le \epsilon,
\end{equation*}
then $x_0\in\RR$, $w_{x_0}(s)\to\theta\kappa(T'(x_0))$ as $s\to \infty$
 and $\left|\arg\tanh T'(x_0)- \arg\tanh d\right|\le C_0 \epsilon$.
\end{prop}
\begin{proof}The original statement of this result was proved in Theorem 3 in \cite{MZjfa07}. The version we are citing comes from (ii) of Proposition 1.1. in \cite{MZisol10}. See this latter paper for the precise justification.
\end{proof}
From Lemma \ref{lemtrapped} and this trapping criterion, we already derived in \cite{MZisol10} the fact that $x$ is non-characteristic and that

\medskip

{\it for all $\epsilon>0$, there exists $\delta>0$ and $L>0$ such that for all $x\in(-\delta,0)$ and $L_k \ge L$, we have 
\begin{equation}\label{pp}
\left|\arg\tanh(T'(x))-\arg\tanh\left({\bar d}^*_1(s_k)\right)\right|
\le \epsilon.
\end{equation}
}

\medskip

\noindent Starting from this estimate and Lemma \ref{lemtrapped}, we still follow the proof of Proposition 3.10 in \cite{MZisol10}, adding however the following new ingredient, which directly follows from estimate \eqref{refequid} proved in Theorem \ref{propref}, and makes the only novelty with respect to \cite{MZisol10}:
\[
1-d_1(s)\sim 2 e^{2(\bar\alpha_1+\zeta_0)}s^{-\gamma_1}\mbox{ as }s\to \infty.
\]
Recall first from \eqref{chapeau0} and the definition \eqref{defgi} of $\gamma_1$ that for $|x|$ small enough, we have
\begin{equation}\label{p0}
\frac 1{C|\log|x||^{\gamma_1}}\le b \le \frac C{|\log|x||^{\gamma_1}}\mbox{ and }
\frac 1{C|\log|x||^{\gamma_1}}\le |T'(x)-1| \le \frac C{|\log|x||^{\gamma_1}}.
\end{equation}
Therefore, from the definitions given in Lemma \ref{lemtrapped} above, we have as $x\to 0$,
\begin{align}
S_k& = - \log |x|-\log(1+e^{-L_k})+O(|\log|x||^{-\gamma_1}),\nonumber\\
1-{\bar d}_1(s_k) &\sim 2 e^{2(\bar\alpha_1+\zeta_0)}|\log|x||^{-\gamma_1},\label{p2}\\
\bar \nu_1(s_k) &= O(|\log|x||^{-\gamma_1}).\label{p3}
\end{align}
Consider then $\epsilon>0$. Since we have by definition \eqref{deflambda} of $\bar \lambda_1$:
\[
{\bar\lambda_1}^{-(p-1)}=\left(1+\frac{\bar \nu_1}{1-\bar d_1}\right)\left(1+\frac{\bar \nu_1}{1+\bar d_1}\right),
\]
using \eqref{pp}, \eqref{p2} and \eqref{p3}, we see that for $|x|$ small and $L_k$ large, we have
\begin{equation*}
\frac{|\bar \nu_1|}{1-\bar d_1}\le C\epsilon.
\end{equation*}
Since we have $1- {\bar d}_1^*=1-{\bar d}_1 +O(\bar \nu_1)$ for small $\bar \nu_1$ from \eqref{defbd1}, using \eqref{p2}, the last line gives for $|x|$ small and $L_k$ large:
\begin{equation}\label{p5}
\left|1- {\bar d}_1^*-2 e^{2(\bar\alpha_1+\zeta_0)}|\log|x||^{-\gamma_1}\right|\le C \epsilon |\log|x||^{-\gamma_1}.
\end{equation}
Therefore, from \eqref{pp} together with \eqref{p5} and \eqref{p0}, we write
for $x<0$, $|x|$ small and $L_k$ large:
\begin{align*}
\left|T'(x) - {\bar d}_1^*\right|&\le \max \left(1-(T'(x))^2, 1-({\bar d}_1^*)^2\right)\left|\arg\tanh(T'(x))-\arg\tanh\left({\bar d}^*_1(s_k)\right)\right|\\
&\le C\epsilon |\log|x||^{-\gamma_1}.
\end{align*}
Since $\epsilon>0$ was arbitrary, using \eqref{p5}, we see that \eqref{cor1} follows when $x<0$. By integration, we get \eqref{cor0}, also when $x<0$.
It remains then to treat the case $x>0$.\\
{\bf Case $x>0$}: Introducing $u^\sharp(x^\sharp,t)=(-1)^ku(-x^\sharp,t)$, we see that $u^\sharp$ is also a solution of \eqref{eq:nlw_u} with $0$ as a characteristic point and that $T^\sharp(x^\sharp) = T(-x^\sharp)$. Thus, we reduce to the study of $u^\sharp$ for $x^\sharp<0$.\\
 Since we have from the definition of the similarity variables' transformation \eqref{def:w} that $w^\sharp_0(y^\sharp,s)=(-1)^kw_0(-y^\sharp,s)$, we derive from \eqref{profilew0} the following estimate (after reversing the order of the solitons): 
\begin{equation*}
\left\|\vc{w^\sharp_0(s)}{\ps w^\sharp_0(s)} - \vc{\ds\sum_{i=1}^{k} (-1)^i\kappa(D_i(s))}0\right\|_{\H} \to 0\mbox{ as }s\to \infty,
\end{equation*}
where $D_i(s) = -d_{k-i}(s)$ satisfies the following from \eqref{refequidw0} and the symmetry relation \eqref{antis} on $\bar \zeta_i(s)$:
\begin{equation*}
D_i(s) = -\tanh \Xi_i(s)\mbox{ and }\Xi_i(s)= - \zeta_{k-i}(s)  = -\bar \zeta_{k-i}(s) - \zeta_0=\bar \zeta_i(s)-\zeta_0.
\end{equation*}
Thus, up to replacing $\zeta_0$ by $-\zeta_0$, we see that we are in the case ``$x<0$'' already treated above, and the result follows for $u^\sharp$ with $x^\sharp<0$, hence for $u$ with $x>0$. This concludes the proof of Corollary \ref{corref}. 
\end{proof}

\section{Construction of a multi-soliton solution in similarity variables}\label{secweak}
In this section, we construct a multi-soliton solution in similarity variables for equation \eqref{eq:nlw_w}. Technically, we use the dynamical system formulation introduced in \cite{MZisol10}. For that reason, we introduce for all $d\in (-1,1)$ and $\nu >-1+|d|$, $\kappa^*(d,\nu,y) = (\kappa_1^*, \kappa_2^*)(d,\nu,y)$ where 
\begin{align}\label{defk*}
\kappa_1^*(d,\nu, y) & =
\kappa_0\frac{(1-d^2)^{\frac 1{p-1}}}{(1+dy+\nu)^{\frac 2{p-1}}}, \\
\kappa_2^*(d,\nu, y) & = \nu \pnu \kappa_1^*(d,\nu, y) =
-\frac{2\kappa_0\nu}{p-1}\frac{(1-d^2)^{\frac 1{p-1}}}{(1+dy+\nu)^{\frac {p+1}{p-1}}}.
\end{align}
We refer to these functions as ``generalized solitons'' or solitons for short. Notice that for any $\mu\in\m R$, $\kappa^*(d,\mu e^s,y)$ is a solution to equation \eqref{eq:nlw_w}. Then\\
- $\kappa^*(d,\mu e^s,y)\to (\kappa(d),0)$ in $\H$ as $s\to -\infty$;\\
- when $\mu=0$, we recover the stationary solutions $(\kappa(d),0)$ defined in \eqref{defkd};\\
- when $\mu>0$, the solution exists for all $(y,s) \in (-1,1)\times \m R$ and converges to $0$ in $\H$ as $s\to \infty$ (it is a heteroclinic connection between $(\kappa(d),0)$ and $0$);\\
- when $\mu<0$, the solution exists for all $(y,s) \in (-1,1)\times \left(-\infty, \log\left(\frac {|d|-1}\mu\right)\right)$ and blows up at time $s=\log\left(\frac {|d|-1}\mu\right)$.\\
We also introduce for $l=0$ or $1$, for any $d\in (-1,1)$ and $r\in \H$, 
\begin{gather}
\pp_l^d(r) =\phi\left(W_l(d), r\right), \quad \text{where} \label{defpdi} \\
\begin{aligned}
\phi(q,r) & := \int_{-1}^1 \left(q_1r_1+q_1' r_1' (1-y^2)+q_2r_2\right)\rho dy \\
& =\int_{-1}^1 \left(q_1\left(-\q L r_1+r_1\right) +q_2 r_2\right)\rho dy,\nonumber\\
W_l(d,y) & := (W_{l,1}(d,y), W_{l,2}(d,y)), \quad \text{with} \nonumber
\end{aligned} \\
W_{1,2}(d,y)(y)= \cc_1(d) \frac{(1-d^2)^{\frac 1{p-1}}(1-y^2)}{(1+dy)^{\frac 2{p-1}+1}},\quad
W_{0,2}(d,y) = \cc_0\frac {(1-d^2)^{\frac 1{p-1}}(y+d)}{(1+dy)^{\frac 2{p-1}+1}}, \label{defWl2-0}
\end{gather}
for some positive $\cc_1(d)$ and $\cc_0$, and $W_{l,1}(d,y)\in \H_0$ is uniquely determined as the solution of 
\begin{equation}\label{eqWl1-0}
-\q L r + r =\left(l - \frac{p+3}{p-1}\right)W_{l,2}(d) - 2 y\py W_{l,2}(d)+ \frac 8{p-1} \frac{W_{l,2}(d)}{1-y^2}
\end{equation}
normalized by the fact that $\pp_l^d(F_l(d)) =\phi\left(W_l(d), F_l(d)\right)$, where
\begin{equation*}
F_1(d,y)=(1-d^2)^{\frac p{p-1}}\vc{\frac{1}{(1+dy)^{\frac 2{p-1}+1}}}{\frac{1}{(1+dy)^{\frac 2{p-1}+1}}},\;\; F_0(d,y)=(1-d^2)^{\frac 1{p-1}}\vc{\ds\frac{y+d}{(1+dy)^{\frac 2{p-1}+1}}}{0}
\end{equation*}
(see estimate (3.57) in \cite{MZajm11} for more details).

\bigskip

Given $k\ge 2$ and $s_0>0$, we will construct the multi-solution as a solution
to the Cauchy problem of equation \eqref{eq:nlw_w} with initial data
\begin{equation}\label{w0}
w(y,s_0)=\sum_{i=1}^k (-1)^i \kappa^*\left(\bar d_i(s_0),\nu_{i,0}\right)\mbox{ with }|\nu_{i,0}|\le s_0^{-\frac 12 - |\gamma_i|},
\end{equation}
where $\bar d_i(s_0)$ is fixed by 
\[
\bar d_i(s_0)= -\tanh \bar \zeta_i(s_0),
\]
$\bar \zeta_i(s_0)$ is defined in \eqref{solpart} and $\gamma_i$ is defined in \eqref{defgi}. Such a solution will be denoted by $w(s_0,(\nu_{i,0})_i,y,s)$, or, when there is no ambiguity, by $w(y,s)$ or $w(s)$ for short. We will show that when $s_0$ is fixed large enough, we can fine-tune the parameters $\nu_{i,0}$ in the intervals $[-s_0^{-\frac 12 - |\gamma_i|},s_0^{-\frac 12 - |\gamma_i|}]$ so that the solution $w(s_0,(\nu_{i,0})_i,y,s)$ 
will decompose as a sum of $k$ decoupled solitons. This is the aim of the section:
\begin{prop}[A multi-soliton solution in the $w(y,s)$ setting]\label{propw} For any integer $k\ge 2$, there exist $s_0>0$, $\nu_{i,0}\in\m R$ for $i=1,\dots,k$ and $\zeta_0\in\m R$ such that equation \eqref{eq:nlw_w} with initial data (at $s=s_0$) given by \eqref{w0} is defined for all $(y,s)\in(-1,1)\times [s_0, \infty)$, satisfies $(w(s), \ps w(s))\in\H$ for all $s\ge s_0$, and 
\begin{equation}\label{cprofile}
\left\|\vc{w(s)}{\ps w(s)} - \vc{\ds\sum_{i=1}^{k} (-1)^{i+1}\kappa(d_i(s))}0\right\|_{\H} \to 0\mbox{ as }s\to \infty,
\end{equation}
for some continuous $d_i(s) = -\tanh \zeta_i(s)$ satisfying 
\begin{equation}\label{equid}
\zeta_i(s) -\bar \zeta_i(s)\to\zeta_0\mbox{ as }s\to\infty\mbox{ for }i=1,...,k
\end{equation}
where the $\bar\zeta_i(s)$ are introduced in \eqref{solpart}.
\end{prop}

\begin{nb}
Note from \eqref{w0} that initial data are in $H^1\times L^2(-1,1)$. 
Going back to the $u(x,t)$ formulation, we see that initial data is also in $H^1\times L^2(-1,1)$ of the initial section of the backward light-cone. Therefore, from the solution to the Cauchy-problem in light-cones, we see that the solution stays in $H^1\times L^2$ of any section.
\end{nb}

\bigskip

As one can see from \eqref{w0}, at the initial time $s=s_0$, $w(y,s_0)$ is a pure sum of solitons. From the continuity of the flow associated with equation \eqref{eq:nlw_w} in $\H$ (this continuity comes from the continuity of the flow associated with equation \eqref{eq:nlw_u} in $H^1\times L^2$ of sections of backward light-cones), $w(y,s)$ will stay close to a sum of solitons, at least for a short time after $s_0$. In fact, we can do better, and impose some orthogonality conditions, killing the zero and expanding directions of the linearized operator of equation \eqref{eq:nlw_w} around the sum of solitons. The following modulation technique from Merle and Zaag in \cite{MZisol10} is crucial for that.
\begin{prop}[A modulation technique; Proposition 2.1 of \cite{MZisol10}]\label{lemode0}
For all $A\ge 1$, there exist $E_0(A)>0$ and $\epsilon_0(A)>0$ such that for all  $E\ge E_0$ and $\epsilon\le \epsilon_0$,  if $v\in \H$ and for all $i=1,...,\k$, $(\hat d_i,\hat \nu_i)\in(-1,1)\times \m R$ are such that
\begin{equation*}
-1+\frac 1A \le \frac{\hat \nu_i}{1-|\hat d_i|}\le A,\;\;
\hat \zeta_{i+1}^*-\hat \zeta_i^*\ge E\mbox{ and }\|\hat q\|_{\H}\le \epsilon
\end{equation*}
where $\hat q = v-\ds\sum_{j=1}^{\k}(-1)^j \kappa^*(\hat d_j,\hat \nu_j)$ and $\hat d_i^*=\frac{\hat d_i}{1+\hat \nu_i} = -\tanh \hat \zeta_i^*$, then, there exist $(d_i,\nu_i)$ such that for all $i=1,\dots,\k$ and $l=0,1$,
\begin{enumerate}
\item $\pp_l^{d_i^*}(q)=0$ where $q:=v-\sum_{j=1}^{\k}(-1)^j \kappa^*(d_j,\nu_j)$,
\item $\ds \left| \frac{\nu_i}{1-|d_i|}- \frac{\hat \nu_i}{1-|\hat d_i|}\right|+|\zeta_i^*-\hat \zeta_i^*|\le C(A)\|\hat q\|_{\H}\le C(A)\epsilon$,
\item $\ds -1+\frac 1{2A} \le \frac{\nu_i}{1-|d_i|}\le A+1$,
$\ds \zeta_{i+1}^*-\zeta_i^*\ge \frac E2$ and $\|q\|_{\H} \le C(A)\epsilon$,
\end{enumerate}
where $d_i^*=\frac{d_i}{1+\nu_i} = -\tanh \zeta_i^*$. 
\end{prop}

Let us apply this proposition with $r=w(y,s_0)$ \eqref{w0}, $\hat d_i=\bar d_i(s_0)$ and $\hat \nu_i=\nu_{i,0}$. Clearly, we have $\hat q=0$. Then, from \eqref{w0}, \eqref{solpart} and straightforward calculations, we see that
\[
\frac{|\hat \nu_i|}{1-|\hat d_i|}\le \frac C{\sqrt{s_0}} \quad \text{and} \quad \hat \zeta_{i+1}^*-\hat \zeta_i^*\ge \frac{(p-1)}4 \log s_0
\]
for $s_0$ large enough. Therefore, Proposition \ref{lemode0} applies with $A=2$ and from the continuity of the flow associated with equation \eqref{eq:nlw_w} in $\H$, we have a maximal $\bar s=\bar s(s_0,(\nu_{i,0})_i)>s_0$ such that $w$ exists for all time $s\in [s_0, \bar s)$ and $w$ can be modulated 
in the sense that
\begin{equation}\label{defq}
w(y,s) = \sum_{i=1}^k (-1)^i \kappa^*(d_i(s), \nu_i(s))+q(y,s)
\end{equation}
where the parameters $d_i(s)$ and $\nu_i(s)$ are such that for all $s\in[s_0, \bar s]$,
\[
\pp_l^{d_i^*(s)}(q(s)) =0,\;\;\forall l=0,1,\;\;i=1,\dots,k
\]
and
\begin{equation}\label{conmod}
\frac{|\nu_i(s)|}{1-|d_i(s)|}\le s_0^{-1/4},\;\;
\zeta_{i+1}^*(s)-\zeta_i^*(s)\ge \frac{(p-1)}8 \log s_0\quad \text{and} \quad \|q(s)\|_{\H}\le \frac 1{\sqrt{s_0}}.
\end{equation}
Two cases then arise:

- either $\bar s(s_0,(\nu_{i,0})_i)=+\infty$;

- or $\bar s(s_0,(\nu_{i,0})_i)<+\infty$ and one of the $\le$ symbol in \eqref{conmod} is a $=$.

At this stage, we see that controlling the solution $w(s)\in \q H$ is equivalent to controlling $q\in \q H$, $(d_i(s))_i\in(-1,1)^k$ and $(\nu_i(s))_i\in\m R^k$. Introducing
\begin{equation}\label{defJ}
J = \sum_{i=2}^k e^{-\frac{2}{p-1} (\zeta_i - \zeta_{i-1})}, \quad \bar J = \sum_{i=1}^k \frac{|\nu_i|}{1-d_i^2}, \quad \hat J = \sum_{i=1}^k e^{- \frac{\bar p}{p-1}(\zeta_i - \zeta_{i-1})}
\end{equation}
where
\begin{equation}\label{defpb}
 \bar p =
\begin{cases}
p & \text{ if } p <2, \\
2 - 1/100 & \text{ if } p =2, \\
2 & \text{ if } p >2,
\end{cases}
\end{equation}
we recall from \cite{MZisol10,MZajm11} 
the dynamical estimates satisfied by those components:
\begin{prop}[Dynamics of the parameters]\label{propdyn}
There exists $\delta>0$
such that for $s_0$ large enough and for all $s\in[s_0,\bar s)$, we have 
\begin{align}
\frac{| \dot \nu_i - \nu_i |}{1-d_i^2} & \le C  \left( \| q \|_{\q H}^2 + J + \| q \|_{\q H} \bar J\right) 
\label{est:nu} \\
\left|\frac {\dot \zeta_i }{c_1(p)} - (e^{-\frac{2}{p-1}  (\zeta_{i}-\zeta_{i-1})} - e^{-\frac{2}{p-1}  (\zeta_{i+1}-\zeta_{i})} )\right| & \le C(\| q \|_{\q H}^2 + (J +\| q \|_{\q H}) \bar J + J^{1+\delta}), \label{est:zeta} \\
\| q(s) \|_{\q H}^2 & \le C e^{-\delta(s-s_0)} \| q(s_0) \|_{\q H}^2 + C \hat J(s)^2, \label{est:q}
\end{align}
where $\zeta_i(s) = - \arg\tanh d_i(s)$, $c_1(p)>0$ was already introduced in \cite[Proposition 3.2]{MZajm11}, $J$, $\bar J$ and $\hat J$ are introduced in \eqref{defJ}.
\end{prop}
\begin{proof} This statement is a small refinement of Claims 3.8 and 3.9 in \cite{MZajm11} and Proposition 3.2 in \cite{MZisol10} where the authors handle the same equation. For that reason, we leave the proof to Appendix \ref{appdyn}.
\end{proof}

From \eqref{est:zeta} and \eqref{conmod}, we see that $(\zeta_i(s))_i$ satisfies a perturbed version of the system \eqref{eq:tl} studied in Section \ref{secref}. Moreover, as one can see from our purpose stated in Proposition \ref{propw}, our aim is to show the existence of a solution with $\zeta_i(s) \sim \bar \zeta_i(s)$ as $s\to \infty$ (at least when $i\neq \frac{k+1}2$). Hence, it is natural to do as in \eqref{defxi0} in Section \ref{secref} and linearize system \eqref{est:zeta} around $(\bar \zeta_i(s))_i$ by introducing
\begin{equation}\label{defxi1}
 \xi_i(s) = \frac 2{p-1}(\zeta_i(s) - \bar \zeta_i(s))
 \end{equation}
If $\boldsymbol{\xi}(s)=(\xi_1(s),\dots, \xi_k(s))$, then we obtain the following perturbed version of system \eqref{eqlin}: For all $s\in [s_0, \bar s)$:
\begin{equation}\label{eqlin1} 
\left|\dot {\boldsymbol \xi}(s) - \frac 1s M \boldsymbol \xi(s)\right|\le \frac Cs |\boldsymbol\xi(s)|^2+C(\| q(s) \|_{\q H}^2 + (J(s) +\| q (s)\|_{\q H}) \bar J(s) + J(s)^{1+\delta}),
\end{equation}
where the self-adjoint $k\times k$ matrix $M$ is introduced in \eqref{defM} and is diagonalizable as stated in Lemma \ref{eigenm} (note that here we keep the time variable $s$ and don't work with $\tau = \log s$). It is then natural to work in the basis defined by its eigenvectors $(\boldsymbol e_i)_i$ by introducing 
$\boldsymbol{\phi}(s) = (\phi_1(s), \dots, \phi_k(s))$ defined by
\begin{equation}\label{defphi}
\boldsymbol \xi(s) = \sum_{i=1}^k \phi_i(s) \boldsymbol e_i.
\end{equation}
Note that thanks to all these changes of variables, controlling $w$ is equivalent to the control of $(q,\boldsymbol \phi,(\nu_i)_i)$. As a matter of fact, in order to control $w$ near multi-solitons, we introduce the following set:
\begin{defi}[Definition of a shrinking set for the parameters]\label{defVa}
We say that 
$w\in \q V(s_0,s)$ if and only if  
\begin{equation}\label{defV}
\begin{array}{>{\ds}l}
s^{1/2+\eta} \| q \|_{\q H} \le 1, \quad \forall i =1, \dots, k, \quad s^{1/2 + |\gamma_i|} |\nu_i|  \le 1, \\
\forall i =2, \dots, k, \quad s^{\eta} |\phi_i |  \le 1, \quad \text{and} \quad s_0^\eta | \phi_1 | \le 1,
\end{array}, 
\end{equation}
where
\begin{equation}\label{defeta} 
\eta = \frac{1}{4} \min \left\{1,\delta, \frac{\bar p}{2} - \frac{1}{2} \right\},
\end{equation}
$\delta>0$ is defined in Proposition \ref{propdyn}  and $\bar p$ is defined in \eqref{defpb}.
\end{defi}
From the existence of $\bar s$, we know that there is a maximal $s^*(s_0,(\nu_{i,0})_i)\in [s_0, \bar s)$ such that for all $s\in[s_0, s^*)$, $w(s) \in \q V(s_0,s)$ and:\\
- either $s^*=+\infty$,\\
- or $s^*<+\infty$ and from continuity, $w(s^*) \in \partial \q V(s_0,s^*)$, in the sense that one $\le$ symbol in \eqref{defV} has to be replaced by the $=$ symbol. 

\medskip

Our aim is to show that for $s_0$ large enough, one can find a parameter $(\nu_{i,0})_i$ in $\prod_{i=1}^k[-s_0^{-\frac 12 - |\gamma_i|},s_0^{-\frac 12 - |\gamma_i|}]$ such that 
\begin{equation}\label{aim}
s^*(s_0,(\nu_{i,0})_i)=+\infty.
\end{equation}
Introducing
\begin{equation}\label{deftj}
\tilde J = \sum_{j=2}^k \phi_j^2
\end{equation}
(we emphasize that the sum's index runs from $2$ to $k$, and not from $1$ to $k$),
we derive from \eqref{est:zeta} the following differential inequality satisfied by $\boldsymbol \phi(s)$:
\begin{cor}[Dynamics for $\phi_i$]\label{cordynphi}
For all $s\in [s_0, s^*)$,
\begin{equation}
\left| \dot \phi_i + \frac{m_i}{s} \phi_i \right| \le C \frac{\tilde J}{s} + C \left( \| q \|_{\q H}^2 + (J +\| q \|_{\q H}) \bar J + J^{1+\delta} \right).
\label{est:phi}
\end{equation}
\end{cor}
\begin{nb}
This corollary is more subtle than one may think. Indeed, if we project the differential inequality \eqref{eqlin1} on the eigenvalues of $M$, then we trivially obtain almost the same identity as \eqref{est:phi}, except that 
the right-hand side has an additional term  $C\frac{\phi_1^2}s$.
With more work, we get the subtle version, which suits our purposes.
\end{nb}
\begin{proof}
This is a direct consequence of the estimate \eqref{est:zeta}. 
First recall from the definitions \eqref{solpart}, \eqref{defalphai} and \eqref{defsi} of $\bar \zeta_i(s)$, $\bar\alpha_i$ and $\sigma_i$ that
\[ 
e^{- \frac{2}{p-1} (\ds \bar \zeta_{i} - \bar \zeta_{i-1})} = \frac{p-1}{2c_1} \frac{\sigma_{i-1}}{s}. 
\]
Then, as $w \in \q V(s_0,s)$, $\xi_i$ defined in \eqref{defxi1} are bounded, we have the expansion (uniform in $s$):
\begin{align*}
e^{- \frac{2}{p-1} (\zeta_{i} - \zeta_{i-1}) } & = e^{- \frac{2}{p-1} (\ds \bar \zeta_{i} - \bar \zeta_{i-1}) -(\xi_i - \xi_{i-1})} \\
& = e^{- \frac{2}{p-1} (\ds \bar \zeta_{i} - \bar \zeta_{i-1})} \left( 1 - (\xi_i - \xi_{i-1}) + O ((\xi_i-\xi_{i-1})^2 \right) \\
&= \frac{p-1}{2c_1} \frac{\sigma_{i-1}}{s}\left( 1 - (\xi_i - \xi_{i-1}) + O ((\xi_i-\xi_{i-1})^2 \right). 
\end{align*}
Hence, from the differential identities \eqref{eq:tl} and \eqref{est:zeta} satisfied by $\bar \zeta_i$ and $\zeta_i$, the equation on $\xi_i$ writes (recall from \eqref{defsi} that $\sigma_0 = \sigma_k =0$)
\begin{align*}
\dot \xi_i & = \frac 2{p-1}(\dot \zeta_i -\dot{\bar {\zeta_i}}) \\
& = - \frac{\sigma_{i-1}}{s}  \left( \xi_i - \xi_{i-1} + O(\xi_i-\xi_{i-1})^2 \right) + \frac{\sigma_{i}}{s} (\xi_i - \xi_{i-1} + O(\xi_i-\xi_{i-1})^2) \\
& \qquad + O(\| q \|_{\q H}^2 + (J +\| q \|_{\q H}) \bar J + J^{1+\delta}) \\
& = \frac{1}{s} \left( \sigma_{i-1} \xi_{i-1} - (\sigma_{i-1} + \sigma_i) \xi_i + \sigma_i \xi_{i+1} + O \left( \sum_{i=2}^k |\xi_i - \xi_{i-1}|^2 \right) \right) \\
& \qquad + O(\| q \|_{\q H}^2 + (J +\| q \|_{\q H}) \bar J + J^{1+\delta}).
\end{align*}
Hence we conclude from the definition \eqref{defM} of the matrix $M$
\[ 
\dot{\boldsymbol \xi} = \frac{1}{s} M\boldsymbol \xi + O \left( \frac{\sum_{i=2}^k |\xi_i - \xi_{i-1}|^2}{s} \right) + O( \| q \|_{\q H}^2 + (J +\| q \|_{\q H}) \bar J + J^{1+\delta} ). 
\]
Note that this differential inequality is already more accurate than \eqref{eqlin1} which was obtained by a rough Taylor expansion of \eqref{est:zeta}.\\
Now if we denote $\boldsymbol e_i = {}^t(e_{i,1}, \dots, e_{i,k})$ the eigenvector of $M$ defined in Lemma \ref{eigenm} (recall that $e_1 = {}^t(1, \dots ,1)$), we see from the definition \eqref{defphi} of $\phi(s)$ that
\[ \xi_i - \xi_{i-1} = \sum_{j=1}^k \phi_j (e_{j,i} - e_{j,i-1}) = \sum_{j=2}^k \phi_j (e_{j,i} - e_{j,i-1}), \]
and from this,  we deduce:
\[ \sum_{i=2}^k |\xi_i - \xi_{i-1}|^2 = O(\tilde J)
 \]
where $\tilde J$ is defined in \eqref{deftj}.
Hence, when projecting the last relation of $\dot {\boldsymbol \xi}$ on the $\boldsymbol e_i$, we obtain the desired relation. This concludes the proof of Corollary \ref{cordynphi}. 
\end{proof}

With 
 Proposition \ref{propdyn} and Corollary \ref{cordynphi} in our hands, we are in a position to prove the following, which directly implies Proposition \ref{propw}:
\begin{prop}[A solution $w(y,s)\in \q V(s_0,s)$]\label{reducw} For $s_0$ large enough, there exists $(\nu_{i,0})_i\in\prod_{i=1}^k[-s_0^{-\frac 12 - |\gamma_i|},s_0^{-\frac 12 - |\gamma_i|}]$ such that equation \eqref{eq:nlw_w} with initial data (at $s=s_0$) given by \eqref{w0} is defined for all $(y,s)\in(-1,1)\times [s_0, \infty)$ and satisfies $w (s) \in \q V(s_0,s)$ for all $s \ge s_0$.
\end{prop}
\begin{proof}[Proof of Proposition \ref{reducw}]
In fact, we started the proof of this proposition right after the statement of Proposition \ref{propw}. For the sake of clarity, we summarize here all the previous arguments, and conclude the proof thanks to a topological argument.

Let $s_0$ be large enough.
Define $\m B$ (resp. $\m S$) the unit ball (resp. sphere) in $(\m R^k, \ell^\infty)$, and the rescaling function
\begin{equation}\label{rescl}
\Gamma_s: \boldsymbol \nu = {}^t (\nu_1, \dots, \nu_k) \mapsto {}^t ( s^{-1/2-|\gamma_1|} \nu_1, \dots,  s^{-1/2-|\gamma_k|} \nu_k), 
\end{equation}
For all $\boldsymbol \nu \in \m B$, we consider the solution $w(s_0,\boldsymbol \nu,y,s)$ (or $w(y,s)$ for short) to the equation \eqref{eq:nlw_w}, with initial condition at time $s_0$ given by \eqref{w0} with 
\[
(\nu_{i,0})_i = \Gamma_{s_0}(\boldsymbol \nu)_i.
\]
As we showed after the statement of Proposition \ref{lemode0}, $w(y,s)$ can be modulated (up to some time $\bar s = \bar s(s_0,\boldsymbol \nu)>s_0$) into a triplet $(q(s), (d_i(s))_i,(\nu_i(s))_i)$. From the uniqueness of such a decomposition (which is a consequence of the application of the implicit function theorem, see the proof of \cite[Proposition 2.1]{MZisol10}), we get
\begin{equation}\label{initmod}
q(s_0) =0, \quad d_i(s_0)=\bar d_i(s_0)\quad \text{and} \quad
\nu_i(s_0) = \Gamma_{s_0}(\boldsymbol \nu)_i. 
\end{equation}
Performing the change of variables \eqref{defxi1} and \eqref{defphi}, we reduce the control of $w(s)$ to the control of $(q(s), (\nu_i(s))_i, (\phi_i(s))_i)$ and we see from \eqref{initmod} that
\begin{equation}\label{initmod2}
\forall i=1,\dots,k, \quad \phi_i(s_0)=0.
\end{equation}
Introduce
\begin{equation}\label{defN}
N(\boldsymbol \nu, s) := \max \left\{ s^{1/2+\eta} \| q(s) \|_{\q H} ,\ \sup_{i} s^{1/2+|\gamma_i|} |\nu_i(s)|,\ \sup_{i \ge 2} s^\eta |\phi_i(s)|,\ s_0^\eta |\phi_1(s)| \right\},
\end{equation}
we see that $\q V(s_0,s)$ (Definition \ref{defVa}) is simply the unit ball of the norm $N(\boldsymbol \nu, s)$.

As asserted in \eqref{aim}, we aim to find $\boldsymbol \nu$ so that the associated $w \in \q C([s_0,\infty), \q H)$ is globally defined for forward times and 
\[ \forall s \ge s_0, \quad N(\boldsymbol \nu,s) \le 1, \quad \text{i.e.} \quad w(s) \in \q V(s_0,s). \] 
We argue by contradiction.  Assume that the conclusion of Proposition \ref{reducw} does not hold. 
In particular, for all $\boldsymbol \nu$, the exit time $s^*(s_0,\boldsymbol \nu)$ is finite, where 
\begin{equation} \label{exit}
s^*(s_0,\boldsymbol \nu)= \sup \{ s \ge s_0 \;|\; \forall \tau \in [s_0,s], \ N(\boldsymbol \nu,\tau) \le 1 \}.
\end{equation}
Then by continuity, notice that 
\begin{equation} \label{Nbord}
N(\boldsymbol \nu, s^*(s_0,\boldsymbol \nu)) =1,
\end{equation}
and that the supremum defining $s^*(s_0,\boldsymbol \nu)$ is in fact a maximum.\\
We now consider the (rescaled) flow for the $\nu_i$, that is
\begin{equation}\label{defPhi}
\Phi: (s,\boldsymbol \nu) \mapsto \Gamma_s^{-1}({}^t(\nu_1(s), \dots , \nu_k(s))). 
\end{equation}
By the properties of the flow, $\Phi$ is a continuous function of $(s,\boldsymbol \nu) \in [s_0, s^*(s_0,\boldsymbol \nu)] \times \m B$.
By definition of the exit time $s^*(s_0,\boldsymbol \nu)$, we have that for all $s \in [s_0, s^*(s_0,\boldsymbol \nu)]$, $\Phi(s, \boldsymbol \nu) \in \m B$.
The following claim allows us to conclude:

\medskip

\noindent\emph{Claim.}
For $s_0$ large enough, we have:\\ 
(i) For all $\boldsymbol \nu\in \m B$, $\Phi(s^*(s_0,\boldsymbol \nu), \boldsymbol \nu) \in \m S$.\\
(ii) The flow $s \mapsto \Phi(s,\boldsymbol \nu)$ is transverse (outgoing) at $s=s^*(s_0, \boldsymbol \nu)$ (when it hits $\m S$).\\ 
(iii) If $\boldsymbol \nu\in \m S$, then $s^*(s_0,\boldsymbol \nu)=s_0$ and $\Phi(s^*(s_0,\boldsymbol \nu), \boldsymbol \nu)=\boldsymbol \nu$.


\begin{proof}[Proof of the Claim] In the following, the constant $C$ stands for $C(s_0)$.\\
(i) Since for all $s\in [s_0,s^*(s_0,\boldsymbol \nu)]$, $N(s_0,s)\le 1$, it follows that $|\boldsymbol \phi(s)|\le C$, hence from  the change of variables \eqref{defxi1} and \eqref{defphi} together with the definition \eqref{solpart} of $\bar \zeta_i(s)$, we see that 
\[ 
|\xi_i(s)| =\frac 2{p-1}|\zeta_i(s) -\bar \zeta_i(s)|\le C
 \text{ so that } |\zeta_i(s) - \zeta_{i-1}(s) - \frac{p-1}{2} \log s| \le C. 
\]
This in turns implies that 
$1/(Cs^{|\gamma_i|}) \le 1 -d_i^2 \le C/s^{|\gamma_i|}$, except for $i = (k+1)/2$ if $k$ is odd, where $1-d_i(s)^2\ge \frac 1C$.
This leads also to the bounds
\[ J \le \frac{C}{s}, \quad \bar J \le \frac{C}{s^{1/2}}, \quad \hat J \le \frac{C}{s^{\bar p/2}}, \quad \tilde J \le \frac{C}{s^{2\eta}}, 
\]
where the different quantities are defined in \eqref{defJ} and \eqref{deftj}.\\
Hence, the estimates \eqref{est:q}, \eqref{initmod}, \eqref{est:nu} and \eqref{est:phi} read as follows: for all $s \in [s_0, s^*(s_0,\boldsymbol \nu)]$
\begin{align}
\| q (s) \|_{\q H} & \le \frac{C}{s^{\bar p/2}}\le \frac{1}{2 s^{1/2+ \eta}}, \text{ and from this} \label{est:q2} \\
| \dot \nu_i - \nu_i | & \le C \left( \frac{1}{s^{|\gamma_i| + \bar p}} + \frac{1}{s^{|\gamma_i| + 1}} + \frac{1}{s^{1/2+ |\gamma_i|+ \bar p/2}}  \right) \le \frac{C}{s^{|\gamma_i| + 1}} \label{est:nu2} \\
\left| \dot \phi_i + \frac{m_i}s \phi_i \right| & \le C \left( \frac{1}{s^{1+2\eta}} + \frac{1}{s^{\bar p}} +\frac{1}{s^{3/2}}+\frac 1{s^{(\bar p+1)/2}}+ \frac{1}{s^{1+\delta}} \right) \le \frac{C}{s^{1+2\eta}}, \label{est:phi2}
\end{align}
provided that $s_0$ is large enough,
(we used the definition \eqref{defeta} of $\eta$ in the first and last line above).
Now, if $i =2, \dots, k$, recall from Lemma \ref{eigenm} and the definition \eqref{defeta} of $\eta$  that $0<2\eta<m_i$. 
Considering $g_i(s) = s^{m_i} \phi_i(s)$, we see that $|\dot g_i(s)| \le C s^{m_i-(1+2\eta)}$. Since $\phi_i(s_0)=0$ by \eqref{initmod2}, we write 
\begin{equation} \label{phibound}
| \phi_i(s)| \le \left(\frac {s_0}s\right)^{m_i}|\phi_i(s_0)| + \frac{C}{s^{2\eta}} = \frac{C}{s^{2\eta}} \le \frac{1}{2 s^\eta}
\end{equation}
for $s_0$ large enough. For $\phi_1$, directly integrating the relation \eqref{est:phi2} and using the fact that $\phi_1(s_0)=0$ (see \eqref{initmod2}) gives, for $s_0$ large enough,
\begin{equation} \label{phi1bound}
| \phi_1(s)| \le |\phi_1(s_0)| + \frac{C}{s_0^{2\eta}} =\frac{C}{s_0^{2\eta}}\le \frac{1}{2s_0^\eta}.
\end{equation} 
Since $N(\boldsymbol \nu, s^*(s_0, \boldsymbol \nu))=1$ by \eqref{Nbord}, we see from the definition \eqref{defN} of $N$ together with \eqref{est:q2}, \eqref{phibound} and \eqref{phi1bound} that necessarily there exists $i=1,\dots,k$ such that
\begin{equation*}
s^*(s_0, \boldsymbol \nu)^{1/2+|\gamma_i|}|\nu_i(s^*(s_0, \boldsymbol \nu))|=1.
\end{equation*}
Using the definitions \eqref{defPhi} and \eqref{rescl} of the flow $\Phi$ and the rescaling function $\Gamma_s$, we get to the conclusion of part (i). 

\medskip
 
\noindent (ii) Assume that $\Phi(s,\boldsymbol \nu) \in \m S$ for some $s\in [s_0,s^*(s_0, \boldsymbol \nu)]$. Therefore, there exists $i=1,\dots,k$ such that
\begin{equation}\label{ibord}
s^{1/2+|\gamma_i|}|\nu_i(s)|=1.
\end{equation}
Using \eqref{est:nu2}, we write 
\begin{align*}
\MoveEqLeft 
\frac{d}{ds} s^{1/2+|\gamma_i|} \nu_i(s) 
 = {s}^{1/2 + |\gamma_i|} \left( \left( \frac{1}{2} + | \gamma_i| \right) \frac{\nu_i(s)}{s} + \dot \nu_i(s) \right) \\
& =  {s}^{1/2 + |\gamma_i|} \left(  \nu_i(s) \left( 1 + \frac{1}{2s} + \frac{|\gamma_i|}{s} \right) + O \left( \frac{1}{{s}^{1+|\gamma_i|}} \right) \right) \\
& =  {s}^{1/2 + |\gamma_i|} \left(  \nu_i(s) + O \left( \frac{1}{{s}^{1+|\gamma_i|}} \right) \right).
\end{align*}
Using \eqref{ibord}, we deduce that for $s_0$ large enough, 
\begin{equation*} 
\frac{d}{ds} s^{1/2+|\gamma_i|} \nu_i(s)\cdot \frac{1}{ {s}^{1/2 +|\gamma_i|} \nu_i(s) } \ge \frac{1}{2}.
\end{equation*}
The same computation holds for any $j$ such that $\nu_j(s^*)$ reaches one extremity of the interval. Thus, the flow is transverse on $\m B$ and part (ii) holds.

\medskip

\noindent (iii) Let $\boldsymbol \nu \in \m S$. From \eqref{initmod} and the definition \eqref{defPhi} of the flow $\Phi$, we see that 
\begin{equation}\label{ini}
\Phi(s_0,\boldsymbol \nu) = \boldsymbol \nu.
\end{equation}
 Since $\boldsymbol \nu \in \m S$, we can use (ii) of the Claim and see that the flow $\Phi$ is transverse to $\m B$ at $s=s_0$. By definition of the exit time, we see that
\[
s^*(s_0, \boldsymbol \nu) = s_0.
\]
Using \eqref{ini}, we get to the conclusion of part (iii). 
\end{proof}
We now concludes the proof of Proposition \ref{reducw}. From  part (ii) of the claim, $\boldsymbol \nu \to s^*(s_0,\boldsymbol \nu)$ is continuous, hence from (i) and (iii),
\[
\boldsymbol \nu \mapsto \Phi(s^*(s_0,\boldsymbol \nu), \boldsymbol \nu)
\]
is a continuous map from $\m B$ to $\m S$ whose restriction to $\m S$ is the identity. By index theory, this is contradiction. Thus, there exists $\boldsymbol \nu \in \m B$ such that for all $s\ge s_0$, $N(s_0, \boldsymbol \nu)\le 1$, hence $w(s_0, \boldsymbol \nu, \cdot,s) \in \q V(s_0, s)$. This is the desired conclusion. 
\end{proof}

It remains to give the proof of Proposition \ref{propw} in order to conclude this section. Let us first recall from Lemma A.2 in \cite{MZisol10} the following continuity result for the family of solitons $\kappa^*(d,\nu)$:
\begin{lem}[Continuity of $\kappa^*$] \label{contk*}For all $A\ge 2$, there exists $C(A)>0$ such that if $(d_1,\nu_1)$ and $(d_2, \nu_2)$ satisfy 
\begin{equation}\label{condA}
\frac {\nu_1}{1-|d_1|},\frac {\nu_2}{1-{|d_2|}}\in [-1+\frac 1A, A],\;\;
\end{equation}
then
\begin{multline}\label{defze}
\|\kappa^*(d_1,\nu_1)-\kappa^*(d_2, \nu_2)\|_{\H} \\
\le C(A)\left(\left| \frac {\nu_1}{1-|d_1|} -\frac {\nu_2}{1-{|d_2|}}\right| +\left|\arg\tanh d_1 - \arg\tanh d_2\right|\right).
\end{multline}
\end{lem}
\begin{nb} Since $\kappa(d,y)=\kappa^*(d,0,y)$ (see \eqref{defkd} and \eqref{defk*}), this statement is a generalization of the continuity identity for the family $\kappa(d,y)$ given in \eqref{contkd}.
\end{nb}

\begin{proof}[Proof of Proposition \ref{propw}]
Let us consider the solution constructed in Proposition \ref{reducw}. Since $w (s) \in \q V(s_0,s)$ for all $s \ge s_0$, from Corollary \ref{cordynphi} and the definition \ref{defVa} of $\q V(s_0,s)$, we see that \eqref{est:phi2} holds. In particular, for $i=1$, we see that
\[
\forall s\ge s_0,\;\;|\phi_1'(s)|\le \frac C{s^{1+2\eta}}.
\]
Therefore, $\phi_1(s)$ converges to some $l_0\in\m R$ as $s\to \infty$. Since $\phi_i(s)\to 0$ for $i=2,\dots,k$, using \eqref{defphi} and the fact that $e_1= ^t(1,\dots,1)$ (see Lemma \ref{eigenm}), we see that $\xi_i(s) \to l_0$. From \eqref{defxi1}, we see that $\zeta_i(s) -\bar \zeta_i(s)\to\zeta_0\equiv\frac{(p-1)}2 l_0$ for all $i=1,...,k$ and \eqref{equid} follows. In particular, 
\[
1-|d_i(s)|\sim C_i s^{-|\gamma_i|}\mbox{ as }s\to \infty,
\]
hence, from the definition \ref{defVa} of $\q V(s_0,s)$, we have
\[
\forall s\ge s_0,\;\; \frac{|\nu_i|}{1-|d_i(s)|}\le C(s_0)s^{-\frac 12}.
\]
Therefore, Lemma \ref{contk*} applies and since $\kappa^*(d_i(s),0,y)=\kappa(d_i(s),y)$, we have
\[
\|\kappa^*(d_i(s),\nu_i(s))-(\kappa(d_i(s),0))\|_{\q H}\le C(s_0)\frac{|\nu_i|}{1-|d_i(s)|}\le C(s_0)s^{-\frac 12}.
\]
As $\|q(s)\|_{\q H}\le \frac C{s^{\frac 12 +\eta}}$ by definition \ref{defVa} of $\q V(s_0,s)$, 
and with \eqref{defq}, we deduce 
\[
\left\|\vc{w(s)}{\ps w(s)} - \vc{\ds\sum_{i=1}^{k} (-1)^{i+1}\kappa(d_i(s))}0\right\|_{\H} \le \|q(s)\|_{\q H} + C(s_0)s^{-\frac 12}\le C(s_0)s^{-\frac 12}
\]
and \eqref{cprofile} follows. This concludes the proof of Proposition \ref{propw}.
\end{proof}


\section{Multi-solitons solution in the $u(x,t)$ setting}\label{seclorentz}
In this section, we use the multi-soliton in similarity variables of Proposition \ref{propw} together with the Lorentz transform to prove Theorem \ref{mainth} and Corollary \ref{cormore}. 

\subsection{Prescribing only one characteristic point}

We prove Theorem \ref{mainth} here. We proceed in 2 parts:\\
- in Part 1, we translate the construction of the previous section into the $u(x,t)$ setting, and recover a solution to our purpose, without the possibility of prescribing the center of mass. This part contains straightforward and obvious arguments which may be skipped by specialists. We give them for the reader's convenience;\\
- in Part 2, we apply the Lorentz transform to the solution constructed in Part 1, making the center of mass of the solitons equal to any prescribed value.

\bigskip

\begin{proof}[Proof of Theorem \ref{mainth}]$ $\\
{\bf Part 1: A multi-soliton solution in the $u(x,t)$ without prescribing the center of mass}

This part has straightforward arguments. It may be skipped by specialists.\\
Consider an integer $k\ge 2$ and consider $w(y,s)$ the solution of \eqref{eq:nlw_w} constructed in Proposition \ref{propw}.\\
Then, let us define $u(x,t)$ as the solution of equation \eqref{eq:nlw_u} with initial data in $\rm H^1_{\rm loc,u}\times \rm L^2_{\rm loc,u}(\m R)$ whose trace in $(-1,1)$ is given by
\begin{equation}\label{u0u1}
u(x,0)=w(x,s_0)\mbox{ and }\partial_t u(x,0)=\partial_s w(x,s_0)+\frac 2{p-1}w(x,s_0)+x\partial_y w(x,s_0).
\end{equation}
We will see that $u(x,t)$ satisfies all the requirements in Theorem \ref{mainth}, except for subscribing the center of mass. More precisely, using the definition of similarity variables' transformation \eqref{def:w} in the other way, we translate the properties of $w(y,s)$ in the following properties of $u(x,t)$:

\medskip

(i) {\it For all $t\in[0,1)$ and $|x|<1-t$, 
\begin{equation}\label{uinside}
u(x,t)=(1-t)^{-\frac 2{p-1}}w\left(\frac x{1-t}, s_0-\log(1-t)\right).
\end{equation}
}
Indeed, by definition \eqref{def:w} of similarity variables, the function on the right-hand side of \eqref{uinside} is a solution to equation \eqref{eq:nlw_u} with the same initial data \eqref{u0u1} as $u(x,t)$. Since that initial data is in  $H^1\times L^2(-1,1)$ and equation \eqref{eq:nlw_u} is well-posed in $H^1\times L^2$ of sections of backward light cones, both solutions are equal from the uniqueness to the Cauchy problem and the finite speed of propagation, hence \eqref{uinside} holds.
In particular, from \eqref{def:w}, we have
\begin{equation}\label{egal}
\forall s\ge 0,\;\;\forall y\in(-1,1),\;\;
w_{0,1}(y,s)=w(y,s+s_0).
\end{equation}


(ii) {\it $u$ is a blow-up solution}. Indeed, if not, then $u$ is global and $u\in L^\infty_{loc}([0,\infty),\rm H^1_{loc,u}\times L^2_{loc,u}(\m R))$. In particular, we write from the Sobolev injection, for all $s\ge 0$ and $\epsilon>0$, 
\begin{equation}\label{tozero}
\|w_{0,1}(s)\|_{L^2_\rho}\le C \|u\|_{L^\infty(|x|<1+\epsilon-t)} e^{-\frac {2s}{p-1}}\to 0\mbox{ as }s\to\infty.
\end{equation}
This is in contradiction with \eqref{egal} and \eqref{cprofile}.

\medskip

(iii) {\it $T(0)=1$}. Indeed, from \eqref{egal} we see that $u(x,t)$ is defined in the cone $|x|<1-t$, $t\ge 0$, hence $T(0)\ge 1$. From \eqref{tozero}, we see that if $T(0)>1+\epsilon$ for some $\epsilon>0$, then the same contradiction follows. Thus $T(0)=1$.

\medskip

(iv) From above, we can use the simplified notation for \eqref{def:w} and write $w_0$ instead of $w_{0,1}$, and rewrite \eqref{egal} as follows:
\begin{equation*}
\forall s\ge 0,\;\;\forall y\in(-1,1),\;\;
w_0(y,s)=w(y,s+s_0).
\end{equation*}
Using \eqref{cprofile} and \eqref{equid}, 
we see that \eqref{cprofile0} follows for $w_0$ with
\[
\zeta_i(s) -\bar \zeta_i(s)\to\zeta_0\mbox{ as }s\to\infty\mbox{ for }i=1,...,k
\]
where $\zeta_0\in \m R$ and $(\bar \zeta_i(s))_i$ \eqref{solpart} is the explicit solution of system \eqref{eq:tl}.
Using the continuity result \eqref{contkd} for $\kappa(d,y)$, we see that \eqref{cprofile0} still holds if we slightly modify the $\zeta_i(s)$ by putting $\zeta_i(s)=\bar \zeta_i(s)+\zeta_0$ as required by \eqref{refequid1}.
Finally, from the classification of the blow-up behavior for general solutions given in page \pageref{classification}, we clearly see that the origin is a characteristic point.\\
Thus, we have a solution obeying all the requirements of Theorem \ref{mainth}, except that we cannot prescribe the center of mass $\zeta_0$ in \eqref{refequid1}.

\bigskip

{\bf Part 2: Prescribing the center of mass of the solitons}


Now, we take the solution constructed in Part 1 and perform a Lorentz transform to be able to prescribe the center of mass of the solitons. 

\medskip

More precisely, given an integer $k \ge 2$, Part 1 gives a blow-up solution $u^\sharp(x^\sharp, t^\sharp)$ of equation \eqref{eq:nlw_u} with $0$ as a characteristic point such that $T^\sharp(0)=1$ and 
\begin{equation}\label{cprofiled}
\left\|\vc{w_0^\sharp(s)}{\ps w_0^\sharp(s)} - \vc{\ds\sum_{i=1}^{k} (-1)^{i+1}\kappa(d_i^\sharp(s))}0\right\|_{\H} \to 0\mbox{ as }s\to \infty,
\end{equation}
where $w_0^\sharp$ is its similarity variables' version around $(0,T^\sharp(0))$ introduced in \eqref{def:w}, 
\begin{equation}\label{dis-location}
d_i^\sharp(s)=-\tanh \zeta_i^\sharp(s),\;\;\zeta^\sharp_i(s) =\bar \zeta_i(s) + \zeta_0^\sharp,
\end{equation}
 $(\bar \zeta_i(s))_i$ is the explicit solution of system \eqref{eq:tl} introduced in \eqref{solpart} and $\zeta_0^\sharp \in \m R$. 
In particular, $u^\sharp$ is defined (at least) in the truncated cone $\q C_{0,1,1}\cap \{t^\sharp\ge 0\}$ defined in \eqref{nonchar}.

\medskip

Given an arbitrary $\zeta_0 \in \m R$, our goal now is to construct $u$ a blow-up solution of equation \eqref{eq:nlw_u} with $0$ as a characteristic point such that its similarity variables' version $w_0$ \eqref{def:w} has a profile decomposing into $k$ solitons as in the following:
\begin{equation}\label{cprofiles}
\left\|\vc{w_0(s)}{\ps w_0(s)} - \vc{\ds\sum_{i=1}^{k} (-1)^{i+1}\kappa(d_i(s))}0\right\|_{\H} \to 0\mbox{ as }s\to \infty,
\end{equation}
with 
\begin{equation}\label{di-location}
d_i(s)=-\tanh \zeta_i(s)\mbox{ and }\zeta_i(s) =\bar \zeta_i(s) + \zeta_0.
\end{equation}
Since this would imply by definitions \eqref{solpart} and \eqref{defalphai} of $\bar \zeta_i(s)$ and $\bar\alpha_i$ that
\[
\frac{\zeta_i(s)+\dots+\zeta_k(s)}k=\zeta_0,
\]
we call this part of the proof ``prescription of the center of mass''.

\medskip

For this, consider Lorentz transforms of $u^\sharp$: given $d \in (-1,1)$ we consider 
\[ 
u(d;x,t) = u^\sharp(x^\sharp,t^\sharp)\mbox{ with }x^\sharp= \frac{x-d(t-1)}{\sqrt{1-d^2}}\mbox{ and }t^\sharp= 1+\frac{t-1-dx}{\sqrt{1-d^2}}. 
\]
Note first that $u(d;x,t)$ is still a solution of equation \eqref{eq:nlw_u}. Note also that the cone $\q C_{0,1,1}$ is preserved by the Lorentz transform, and that the image of the truncated cone $\q C_{0,1,1}\cap \{t^\sharp\ge 0\}$ is included in the truncated cone  $\q C_{0,1,1}\cap \{t\ge 1-\sqrt{\frac{1-|d|}{1+|d|}}\}$.\\
%
%
%
Now, in self-similar variables, the Lorentz transform reads as follows:
\begin{equation}\label{wlorentz}
w_0(d;y,s) : = {\q T}_d w_0^\sharp (y,s) = \frac{(1-d^2)^{\frac 1{p-1}}}{(1+dy)^{\frac 2{p-1}}}w^\sharp(y^\sharp, s^\sharp),\;\;y^\sharp= \frac{y+d}{1+dy},\;\;
s^\sharp= s + \log \frac{\sqrt{1-d^2}}{1+dy}. 
\end{equation}
The following claim allows us to conclude:
\begin{claim}We have the following:\\
(i) It holds that
\[ 
\sup_{|y|<1} \left|(1-{y}^2)^{\frac 1{p-1}}\left(w_0(y,s) - \sum_{i=1}^k (-1)^{i+1} \kappa(d*d_i^\sharp(s), y) \right)\right| \to 0 \quad \text{as} \quad  s \to \infty,
\]
where $d_1 * d_2= \frac{d_1 + d_2}{1+ d_1 d_2}$.\\
(ii) There exists $t_d>1-\sqrt{\frac{1-|d|}{1+|d|}}$ such that $(u(t_d),\partial_tu(t_d))\in H^1\times L^2(|x|<1-t_d)$.
\end{claim}
Indeed, let us define from $u$ a solution $\hat u$ satisfying the requirements of Theorem \ref{mainth}. From translation invariance of equation \eqref{eq:nlw_u}, we may take the time origin at $t=t_d$.\\
Take 
\begin{equation}\label{defd}
d = \tanh(\zeta_0^\sharp - \zeta_0).
\end{equation}
From (ii), we can consider $\hat u(x,t)$ the solution of equation \eqref{eq:nlw_u} with initial data (at $t=t_d$) in $\rm H^1_{\rm loc,u}\times \rm L^2_{\rm loc,u}(\m R)$ whose trace in $(-1,1)$ is given by
\[
\hat u(x,t_d) = u(d;x,t_d)\mbox{ and }\partial_t\hat u(x,t_d) = \partial_t u(d;x,t_d).
\]
Following Part 1, we see that $\hat u$ is a blow-up solution, $\hat T(0)=1$ and
\[
\forall t\in [t_d,1)\mbox{ and }|x|<1-t,\;\;\hat u(x,t) = u(d;x,t).
\]
In particular, 
\[
\forall s\ge 0\mbox{ and }|y|<1,\;\;\hat w_0(y,s)=w_0(y,s)
\]
and (i) of the Claim provides us with an asymptotic expansion for $\hat w_0$ in $L^\infty$ with the weight $(1-y^2)^{\frac 1{p-1}}$ on the one hand. On the other hand, using the classification of all blow-up solutions of equation \eqref{eq:nlw_u} given in page \pageref{classification}, we see that
\begin{equation}\label{cprofileb}
\left\|\vc{\hat w_0(s)}{\ps \hat w_0(s)} - \hat \theta_1\vc{\ds\sum_{i=1}^{\hat k} (-1)^{i+1}\kappa(\hat d_i(s))}0\right\|_{\H} \to 0\mbox{ as }s\to \infty,
\end{equation}
with $\hat \theta_1=\pm 1$ and:\\
- either $0$ is a non-characteristic point, hence $\hat k=1$, $\hat d_1(s)\equiv \hat T'(0)$,\\
- or $0$ is a characteristic point, with $\hat k\ge 2$, and $\hat d_i(s)=-\tanh \hat \zeta_i(s)$ with
\[
\left|\hat \zeta_i(s)-\bar \zeta_i(s)\right|\le C
\]
for $s$ large enough (use the definition \eqref{solpart} of $\bar \zeta_i(s)$ to derive this from \eqref{equid00}). Recalling the following Hardy-Sobolev inequality from Lemma 2.2 page 51 in \cite{MZjfa07}:
\begin{equation}\label{hs}
\forall h\in \H_0,\;\; \|h(1-y^2)^{\frac 1{p-1}}\|_{L^\infty(-1,1)}\le  C\|h\|_{\H_0},
\end{equation}
we write from \eqref{cprofileb} that
\begin{equation}\label{profile-hs}
\sup_{|y|<1} \left|(1-{y}^2)^{\frac 1{p-1}}\left(\hat w_0(y,s) - \hat \theta_1\sum_{i=1}^{\hat k} (-1)^{i+1} \kappa(\hat d_i(s), y) \right)\right| \to 0 \quad \text{as} \quad  s \to \infty,
\end{equation}
Introducing
\[
\hat W_0(\xi,s) = (1-y^2)^{\frac 1{p-1}}\hat w_0(y,s)\mbox{ with }y=\tanh \xi
\]
and $\hat \kappa_0(\xi) = \kappa_0\cosh^{-\frac 2{p-1}} \xi$, we write from (i) of the Claim and \eqref{profile-hs} two expansions of $\hat w_0$ as $s\to \infty$ as follows:
\begin{eqnarray*}
&&\sup_{\xi\in\m R}|\hat W_0(\xi,s)- \sum_{i=1}^k(-1)^{i+1}\hat \kappa_0(\xi-\zeta_i^\sharp(s) +\arg\tanh d)|\to 0,\\
&&\sup_{\xi\in\m R}|\hat W_0(\xi,s)- \hat \theta_1\sum_{i=1}^{\hat k}(-1)^{i+1}\hat \kappa_0(\xi-\hat \zeta_i(s))|\to 0\mbox{ as }s\to \infty.
\end{eqnarray*}
Comparing these two expansions for the same function, we immediately see that
\[
\hat k =k,\;\;\hat \theta_1=1\mbox{ and }\hat \zeta_i(s)=\zeta_i^\sharp(s) -\arg\tanh d+o(1)\mbox{ as }s\to \infty.
\]
In particular, $\hat k\ge 2$, hence $0$ is a characteristic point for $\bar u$. Moreover, using the continuity estimate \eqref{contkd} on $\kappa(d,y)$, we see that \eqref{cprofileb} holds also with
\[
\hat \zeta_i(s)=\zeta_i^\sharp(s) -\arg\tanh d=\bar \zeta_i(s)+\zeta_0^\sharp-\arg\tanh d=\bar \zeta_i(s)+\zeta_0.
\]
where we used the definitions \eqref{dis-location} and \eqref{defd} of $\zeta_i^\sharp(s)$ and $d$ in this last line. This is the desired estimate in Theorem \ref{mainth}. It remains to prove the claim in order to conclude. 

\begin{proof}[Proof of the Claim]$ $\\
(i) Using the Hardy-Sobolev inequality of \eqref{hs}, we see that \eqref{cprofiled} yields the fact that
\begin{equation}\label{esths}
\sup_{s^\sharp\ge \tau^\sharp, |y^\sharp|<1} \left|(1-{y^\sharp}^2)^{\frac 1{p-1}}\left(w_0^\sharp(y^\sharp,s^\sharp) - \sum_{i=1}^k (-1)^{i+1} \kappa(d_i^\sharp(s^\sharp), y^\sharp) \right)\right| \to 0 \quad \text{as} \quad  \tau^\sharp \to \infty. 
\end{equation}
In the following, we will apply the Lorentz transform in the $w$ version \eqref{wlorentz} to this estimate to get the desired conclusion.\\
Note first that straightforward calculations give the fact that $\q T_d$ has a group structure, in the sense that
\[
{\q T}_{d_1} \circ {\q T}_{d_2} ={\q T}_{d_1 * d_2} \mbox{ where }d_1 * d_2= \frac{d_1 + d_2}{1+ d_1 d_2}.
\]
Therefore, since we have $\kappa(d) = {\q T}_d(\kappa_0)$ from the definition \eqref{defkd} of $\kappa(d,y)$, we see that
\begin{equation}\label{kd1d2}
\kappa(d*d_i^\sharp(s),y) = {\q T}_{d * d_i^\sharp(s)} (\kappa_0)={\q T}_{d} \circ {\q T}_{d_i^\sharp(s)} (\kappa_0) = 
{\q T}_{d} \kappa(d_i^\sharp(s)). 
\end{equation}
%
Since we have from \eqref{wlorentz} the fact that
\[
\frac{(1-y^2)(1-d^2)}{(1+dy)^2}=1-{y^\sharp}^2,
\]
we write from \eqref{wlorentz} and \eqref{kd1d2} for $s\ge -\log (1-t_d)$ and $|y|<1$, 
\begin{align}
&\left|(1-{y}^2)^{\frac 1{p-1}}\left(w_0(y,s) - \sum_{i=1}^k (-1)^{i+1} \kappa(d*d_i^\sharp(s), y) \right)\right|\nonumber\\
& = \left|(1-{y}^2)^{\frac 1{p-1}}\left({\q T_d}w_0^\sharp(y,s) - \sum_{i=1}^k (-1)^{i+1} {\q T_d}\kappa(d_i^\sharp(s), y) \right)\right|\nonumber\\
& = \left|\left(\frac{(1-{y}^2)(1-d^2)}{(1+dy)^2}\right)^{\frac 1{p-1}}\left(w_0^\sharp(y^\sharp,s^\sharp) - \sum_{i=1}^k (-1)^{i+1} \kappa(d_i^\sharp(s), y^\sharp) \right)\right|\nonumber\\
& \le \left|(1-{y^\sharp}^2)^{\frac 1{p-1}}\left(w_0^\sharp(y^\sharp,s^\sharp) - \sum_{i=1}^k (-1)^{i+1} \kappa(d_i^\sharp(s^\sharp), y^\sharp) \right)\right|\nonumber\\
& \qquad +\sum_{i=1}^k\left|(1-{y^\sharp}^2)^{\frac 1{p-1}}\left(\kappa(d_i^\sharp(s^\sharp), y^\sharp) -\kappa(d_i^\sharp(s), y^\sharp) \right)\right|,\label{int}
\end{align}
where $y^\sharp$ and $s^\sharp$  defined in \eqref{wlorentz} satisfy 
\begin{equation}\label{dist}
|s^\sharp-s|\le \frac 12 \log\frac{1+|d|}{1-|d|}.
\end{equation}
Using \eqref{hs} and the continuity relation \eqref{contkd} of $\kappa(d,y)$, we write
\begin{align*}
\sup_{|y^\sharp|<1}\left|(1-{y^\sharp}^2)^{\frac 1{p-1}}\left(\kappa(d_i^\sharp(s^\sharp), y^\sharp) -\kappa(d_i^\sharp(s), y^\sharp) \right)\right| & \le \left\|\kappa(d_i^\sharp(s^\sharp)) -\kappa(d_i^\sharp(s)) \right\|_{\q H_0}\\
& \le C|\arg\tanh d_i^\sharp(s^\sharp)-\arg\tanh d_i^\sharp(s)|.
\end{align*}
From the definition \eqref{dis-location} of $d_i^\sharp(s)$ and \eqref{dist}, we see that
\begin{equation}\label{cont-sur}
\sup_{|y^\sharp|<1}\left|(1-{y^\sharp}^2)^{\frac 1{p-1}}\left(\kappa(d_i^\sharp(s^\sharp), y^\sharp) -\kappa(d_i^\sharp(s), y^\sharp) \right)\right|
\le \frac {C|s^\sharp-s|}s
\le \frac{C(d)}s.
\end{equation}
Therefore, using \eqref{int} and \eqref{cont-sur}, we write 
\begin{align*}
&\sup_{|y|<1} \left|(1-{y}^2)^{\frac 1{p-1}}\left(w_0(y,s) - \sum_{i=1}^k (-1)^{i+1} \kappa(d*d_i^\sharp(s), y) \right)\right|\\
& \le \sup_{|y^\sharp|<1,\;\;s^\sharp\ge s-\frac 12 \log\frac{1+|d|}{1-|d|}}
\left|(1-{y^\sharp}^2)^{\frac 1{p-1}}\left(w_0^\sharp(y^*,s^*) - \sum_{i=1}^k (-1)^{i+1} \kappa(d_i^\sharp(s^\sharp), y^\sharp) \right)\right|+ \frac{C(d)}s.
\end{align*}
Using \eqref{esths}, we conclude the proof of (i) of the Claim.

\medskip

(ii) It is enough to prove that for some $s_2>s_1> 0$, we have
\begin{equation}\label{reduc}
I_{d,s_1,s_2} := \int_{s_1}^{s_2}\iint(w_0(y, s))^2+(\py w_0(y, s))^2+(\ps w_0(y, s))^2dy ds\le C(s_2,s_1,d).
\end{equation}
Indeed, if this is true, then, by the mean value theorem, there exists $s_d\in (s_1,s_2)$ such that
\[
\iint(w_0(y, s_d))^2+(\py w_0(y, s_d))^2+(\ps w_0(y, s_d))^2dy=\frac 1{s_2-s_1} I_{d,s_1,s_2} \le \frac{C(s_2,s_1,d)}{s_2-s_1}.
\]
Using the similarity transformation \eqref{def:w} in the other way, we get the desired estimate with $t_d = 1-e^{-s_d}$. Let us prove \eqref{reduc} then.\\
From the transformation \eqref{wlorentz}, \eqref{dist} and the similarity variables definition \eqref{def:w}, we see that
\begin{align*}
I_{d,s_1,s_2}&\le C(d) \int_{s_1-\frac 12 \log\frac{1+|d|}{1-|d|}}^{s_2+\frac 12 \log\frac{1+|d|}{1-|d|}}\iint(w_0^\sharp(y^\sharp, s^\sharp))^2+(\py w_0^\sharp(y^\sharp, s^\sharp))^2+(\ps w_0^\sharp(y^\sharp, s^\sharp))^2dy^\sharp ds^\sharp\\
&\le C(d,s_2,s_1)\int_{t_1(d)}^{t_2(d)}\int_{|x^\sharp|<1-t^\sharp}(u^\sharp(x^\sharp, t^\sharp))^2+ (\partial_x u^\sharp(x^\sharp, t^\sharp))^2+(\partial_t u^\sharp(x^\sharp, t^\sharp))^2.
\end{align*}
Since initial data for $u^\sharp$ is in $H^1\times L^2(-1,1)$ and equation \eqref{eq:nlw_u} is well-posed in $H^1\times L^2$ of sections of the backward light cone with vertex $(0,1)$ (see the paragraph right after \eqref{uinside}), this latter integral is bounded in terms of $d$, $s_1$ and $s_2$. This concludes the proof of the Claim.
\end{proof}
Since the Claim implies Theorem \ref{mainth}, this concludes the proof of Theorem \ref{mainth} too.
\end{proof}

\subsection{Prescribing more characteristic points}
We use the finite speed of propagation to derive the multiple characteristic points case (Corollary \ref{cormore}) from the one characteristic point case (Theorem \ref{mainth}).

\medskip

\begin{proof}[Proof of Corollary \ref{cormore}] Let us first remark that thanks to the invariance of equation \ref{eq:nlw_u} under space and time translations together with the following dilation
\[
\lambda \mapsto u_\lambda(\xi, \tau) = \lambda^{\frac 2{p-1}}u(\lambda \xi, \lambda \tau),
\]
we can prescribe the characteristic point and the blow-up time, in addition to the number of solitons and the center of mass. More precisely, given $\bar x\in \m R$, $\bar T>0$, $\bar k \ge 2$ and $\bar \zeta_0\in \m R$, there exists a blow-up solution $u_{\bar x, \bar T, \bar k, \bar \zeta}$ of equation \eqref{eq:nlw_u} in $\rm H^1_{\rm loc,u}\times \rm L^2_{\rm loc,u}(\m R)$ such that $\bar x$ is a characteristic point, $T(\bar x)=\bar T$ and $w_{\bar x}$ behaves as in \eqref{cprofile0} and \eqref{refequid1} with $k=\bar k$ and $\zeta_0=\bar \zeta_0$.

\medskip

Let us now consider $I=\{1,...,n_0\}$ or $I=\m N$ and for all $n\in I$, $x_n\in \m R$, $T_n>0$, $k_n \ge 2$ and $\zeta_{0,n}\in \m R$ such that
\begin{equation*}
x_n+T_n<x_{n+1}-T_{n+1}.
\end{equation*}
 From this condition, we can define a solution $u(x,t)$ of equation \eqref{eq:nlw_u} in $\rm H^1_{\rm loc,u}\times \rm L^2_{\rm loc,u}(\m R)$ by taking its initial data such that
\[
\forall n\in I,\;\;
\forall x\in (x_n-T_n, x_n+T_n),\;\;u(x,0)=u_{x_n, T_n, k_n, \zeta_{0,n}}(x,0)
\]
and the same for time derivatives. From the finite speed of propagation, this identity propagates in the cone $\q C_{x_n,T_n,1}$ for positive times, in the sense that
\[
\forall n\in I,\;\; \forall t\in [0,T_n),\;\;
\forall x\in (x_n-(T_n-t), x_n+(T_n-t)),\;\;u(x,t)=u_{x_n, T_n, k_n, \zeta_{0,n}}(x,t)
\]
and the same for time derivatives. As we did in Part 1 of the proof of Theorem \ref{mainth} above, we see by obvious arguments that $u(x,t)$ satisfies all the requirements of Corollary \ref{cormore}.
\end{proof}

\appendix

\section{Lyapunov's Theorem}\label{applyap}

We give the statement of the version of Lyapunov's Theorem that we crucially use in the proof of Theorem \ref{propref} given in Section \ref{secref} above.
\begin{thm}[Lyapunov's Theorem]
Let $K$ be a compact set of $\m R^k$, $X : K \to \m R^k$ be a vector field, and $x_0 \in \mathring{K}$, the interior of $K$. Denote by $(t,x) \mapsto \varphi(t,x)$ the flow of $X$ (at time $t$, starting at point $x$ at time 0).\\ 
Assume that $K$ is stable by the flow (in particular, for all $x \in K$, the flow is globally defined), and that there exists $L: K \to \m R$, a continuous function (Lyapunov) such that
\[ \forall x \in K \setminus \{ x_0 \}, \quad t \mapsto L(\varphi(t,x)) \text{ is (strictly) decreasing}.  \] 
Then, $x_0$ is a critical point for $X$ (the only one in $K$), and for all $x \in K \setminus \{ x_0 \}$, $L(x) > L(x_0)$ (so that $L$ reaches its infimum on $K$ at $x_0$ only).\\ 
Furthermore, for all $x \in K$, $\varphi(t,x) \to x_0$ as $t \to +\infty$ ($x_0$ is a global attractor in $K$).
\end{thm}

\begin{nb}
The stability of $K$ can follow from various assumptions on $L$, for example if $L$ is defined on a neighbourhood of $K$ where the decreasing assumption holds on, and $K = L^{-1}((-\infty,\ell])$ for some $\ell \in \m R$.
\end{nb}
\begin{proof}
Let $c= \inf \{ L(x)\; |\; x \in K \}$ and $I= L^{-1}(\{c\})$ be the set of points where $L$ reaches its infimum. $I \ne \varnothing$ because $L$ is continuous and $K$ is compact. Now if $x \in K \setminus \{ x_0 \}$, then $L(\varphi(1,x)) < L(\varphi (0,x)) =L(x)$, so that $x \notin I$. Hence $I = \{ x_0 \}$ and for all $x \in K \setminus \{ x_0 \}$, $L(x) > L(x_0)$. 

\medskip

We now prove that $x_0$ is a critical point. We claim that it is enough to prove that
\begin{equation}\label{goal}
\mbox{there exists }t_0>0\mbox{ such that }\forall t\in[0,t_0],\;\varphi(t,x_0)=x_0.
\end{equation}
Indeed, if \eqref{goal} is true, then by the uniqueness in the Theorem of Cauchy-Lipschitz, we see that $\varphi(t,x_0) = x_0$ for all $t\ge 0$ and $x_0$ is a critical point. Let us then prove \eqref{goal}.\\
Assume by contradiction that \eqref{goal} does not hold. Then, there exists a decreasing sequence of times $t_n \to 0$ such that $\varphi(t_n,x_0) \ne x_0$. As $\varphi(t_{n-1},x_0) = \varphi(t_{n-1}-t_n, \varphi(t_n,x_0))$, the sequence $L(\varphi(t_n,x_0))$ is strictly increasing, hence, it has a limit $a > \varphi(t_1,x_0) > c$. But $t_n \to 0$ so that $\varphi(t_n,x_0) \to x_0$ by continuity of the flow, and by continuity of $L$, $L(\varphi(t_n,x_0)) \to L(x_0) =c$, and we reached a contradiction. Thus, \eqref{goal} holds and $x_0$ is a critical point.

\medskip

Let us finally prove that $\varphi(t,x) \to x_0$. Note that it is enough to prove that 
\begin{equation}\label{but}
L(\varphi(t,x)) \to c\mbox{ as }t\to \infty.
\end{equation}
Indeed, if a sequence $(y_n) \subset K$ is such that $L(y_n) \to c$, then $y_n \to x_0$, otherwise, there exists a subsequence $z_n$ of $y_n$ and $\e >0$ such that for all $n$, $z_n \in K \setminus B(x_0,\e)$. This latter set is compact, so that up to a subsequence which we also denote $z_n$, $z_n$ converges to some $z \in K \setminus B(x_0,\e)$. By continuity, $L(z)=c=L(x_0)$, so that $z=x_0$ from the fact that $I=\{x_0\}$, and we reached a contradiction. Let us then prove \eqref{but}.\\
Assume by contradiction that \eqref{but} does not hold. Then, there exists $\delta>0$ such that the non-increasing function 
\begin{equation}\label{conv}
L(\varphi(t,x))\to c+\delta\mbox{ as }t \to +\infty.
\end{equation}
 As $(\varphi(t,x))_t$ remains in $K_\delta := K \cap L^{-1}([c+\delta,\infty))$ which is a compact, there exists an increasing sequence of times $t_n \to +\infty$ and $\bar x \in K_\delta$ such that $\varphi(t_n,x) \to \bar x$. Note in particular that 
\begin{equation}\label{diff}
\bar x\ne x_0
\end{equation}
Now let $t \in \m R$ and consider the flow starting from $\bar x$. By continuity of the flow, $\varphi(t,\bar x) = \lim_n \varphi(t,\varphi(t_n,x)) = \lim_n \varphi(t+t_n,x)$. As $L$ is continuous, $L(\varphi(t+t_n,x)) \to L(\varphi(t,\bar x))$ on the one hand. On the other hand, from \eqref{conv}, we have $L(\varphi(t+t_n,x)) \to c+\delta$. Hence, for any $t\in\m R$, $L(\varphi(t,\bar x))=c+\delta$. This is a contradiction because we are not on the stationary trajectory (see \eqref{diff}) and $L$ is strictly decreasing outside that trajectory. Thus,  $\varphi(t,x) \to x_0$ as $t\to \infty$. This concludes he proof of Lyapunov's Theorem.
\end{proof}

\section{Dynamics of equation \eqref{eq:nlw_w} near multi-solitons}\label{appdyn}

This section is devoted to the proof of Proposition \ref{propdyn}. Since the proof needs only minor refinements with respect to the proofs of Claims 3.8 and 3.9 in \cite{MZisol10} and Proposition 3.2 in \cite{MZajm11}, we only give indications on the refinements. Hence, this section is not self-contained, since making it self-contained would add many pages with no new techniques with respect to \cite{MZajm11} and \cite{MZisol10}.

\begin{proof}[Proof of Proposition \ref{propdyn}] We first recall from \cite[Appendix C]{MZisol10} the equation satisfied by $q$ defined in \eqref{defq} for all $s\in [s_0, \bar s)$:
\begin{align}
\ds\frac \partial {\partial s}
\vc{q_1}{q_2} & =\ll \vc{q_1}{q_2}
-\sum_{j=1}^\k(-1)^j\left[(\nu_j'(s)-\nu_j(s))\pnu\kappa^*+d_j'(s)\partial_d \kappa^*\right](d_j(s),\nu_j(s),y)\nonumber\nonumber \\
& \qquad +\vc{0}{R} + \vc{0}{f(q_1)}
\label{eqq*}
\end{align}
where
\begin{align*}
\ll\vc{q_1}{q_2} & = \vc{q_2}{\L q_1+\psi q_1-\frac{p+3}{p-1}q_2-2y\py q_2},\\
\psi(y,s)& = p|K^*_1(y,s)|^{p-1} -\frac{2(p+1)}{(p-1)^2}, \qquad
K^*_1(y,s) = \sum_{j=1}^\k  (-1)^i\kappa^*_1(d_j(s),\nu_j(s),y), \\
f(q_1) & = |K^*_1+q_1|^{p-1}(K^*_1+q_1)- |K^*_1|^{p-1}K^*_1- p|K^*_1|^{p-1} q_1, \\
R & = |K^*_1|^{p-1}K^*_1- \sum_{j=1}^\k (-1)^j\kappa^*_1(d_j(s),\nu_j(s),y)^p.
\end{align*}
We now proceed to the justification of the 3 estimates of Proposition \ref{propdyn}, based on Lemma C2., Claims 3.8 and 3.9 of \cite{MZisol10}, together with Proposition 3.2 in \cite{MZajm11}.\\ 
Since estimate \eqref{conmod} holds for all $s\in[s_0, \bar s)$, 
 all those results hold provided that $s_0$ is large enough. As a matter of fact, we take below $s\in[s_0, \bar s)$ and $s_0$ large enough.\\
- \eqref{est:nu} follows from \cite[Lemma C.2 (i)]{MZisol10} (use the last identity of the proof of (i))\\ 
- \eqref{est:q} follows from \cite[Claim 3.8 (ii), (iii)]{MZisol10} (see the proof of Claim 3.9 (ii)there).\\
- With respect to the analysis in \cite{MZisol10}, \eqref{est:zeta} needs some refinements, 
see \cite{MZajm11}. Note first that we have a rough estimate from  \cite[Lemma C.2 (i)]{MZisol10} which we recall:
\begin{equation}\label{rough}
\frac{|d_i'(s)|}{1-d_i(s)^2}=|\zeta_i'(s)| \le C  \left( \| q \|_{\q H}^2 + J + \| q \|_{\q H} \frac{|\nu_i|}{1-d_i^2} \right).
\end{equation}
Using the proof of this statement in that paper and \cite[Appendix C]{MZajm11}, we derive 
\begin{align} \label{2proj}
\left|(-1)^{i+1}d_i'\pp_0^{d_i^*}(\partial_d \kappa^*(d_i,\nu_i)) + \pp_0^{d_i^*}((0,R))\right| \le C \left( \| q \|_{\q H}^2 +J^{1+\delta_1}+ \| q \|_{\q H} \frac{|\nu_i|}{1-d_i^2} \right)
\end{align}
for some $\delta_1>0$. 
The term $\pp_0^{d_i^*}(\partial_d \kappa^*(d_i,\nu_i))$ has been evaluated in \cite[Claim 2.2]{MZisol10}, but we need to further refine it, given the fact that $\frac{|\nu_i|}{1-|d_i|}$ is small (see \eqref{conmod}). Using estimates (2.36) and (2.27) given in the proof of \cite[Claim 2.2]{MZisol10}, we see that 
\begin{multline}
\frac 1{c_0L_i}\pp_0^{{d_i^*}}(\partial_d \kappa^*(d_i,\nu_i)) = \label{sonia}\\
-\frac 4{p-1}\iint Y^2(1-Y^2)^{\frac 2{p-1}-1}dY+(1-x_i) \left(x_id_i^2+\frac{p+1}{p-1}\right)\iint \frac{Y^2(1-Y^2)^{\frac 2{p-1}}}{1-x_i^2d_i^2Y^2} dY
\end{multline}
where $c_0>0$,
\[
L_i =\frac{2\kappa_0(1-d_i^2)^{\frac 1{p-1}-1}(1+\nu_i)^{-\frac {p+1}{p-1}}}{(p-1)(1-{d_i^*}^2)^{\frac 1{p-1}}} \quad \text{and} \quad x_i=\frac 1{\nu_i+1}.
\]
Since $x_id_i=\frac{d_i}{1+\nu_i}=d_i^*$, we write
\begin{equation}\label{e1}
\iint \frac{Y^2(1-Y^2)^{\frac 2{p-1}}}{1-x_i^2d_i^2Y^2} dY\le \frac 1{1-{d_i^*}^2}\iint Y^2(1-Y^2)^{\frac 2{p-1}} dY.
\end{equation}
Using \eqref{conmod} and the definition \eqref{defJ} of $\bar J$, we see that for $s_0$ large enough, we have
\begin{equation}\label{e2}
\frac 1{1-{d_i^*}^2}\le \frac C{1-d_i^2} \quad \text{and} \quad |L_i-\frac{2\kappa_0}{(p-1)(1-d_i^2)}|\le C\frac{|\nu_i|}{1-d_i^2}\le C\bar J.
\end{equation}
Using \eqref{sonia}, \eqref{e1} and \eqref{e2}, we see that
\begin{equation}\label{term1}
\left|\pp_0^{{d_i^*}}(\partial_d \kappa^*(d_i,\nu_i))+\frac {8\kappa_0c_0}{(p-1)^2(1-d_i^2)}\iint Y^2(1-Y^2)^{\frac 2{p-1}-1}dY\right|\le C\bar J.
\end{equation}
Now, we estimate the second term of \eqref{2proj}. Proceeding as for the proof of \cite[Proposition 3.2]{MZajm11} given in Section 3.3 of that paper, we derive the fact that 
\begin{equation}\label{Rproj1}
\left|\pp_0^{d_i^*}((0,R))-c_2(p)(-1)^i\lambda_i^{p-1}\left[\lambda_{i-1}e^{-\frac 2{p-1}(\zeta_i-\zeta_{i-1})}-\lambda_{i+1}e^{-\frac 2{p-1}(\zeta_{i+1}-\zeta_i)}\right]\right|\le CJ^{1+\delta_2}
\end{equation}
for some $c_2(p)>0$, $\delta_2>0$ and $\lambda_j(s) =\frac{(1-d_j(s)^2)^{\frac 1{p-1}}}{[(1+\nu_j(s))^2-d_j(s)^2]^{\frac 1{p-1}}}$,  where by convention $\zeta_0(s)\equiv-\infty$ and $\zeta_{k+1}(s)\equiv +\infty$. Since we have from \eqref{conmod} and \eqref{defJ}, $|\lambda_j(s) -1|\le C\frac{|\nu_j(s)|}{1-|d_j(s)|}\le C\bar J(s)$ for $s_0$ large enough, we see from \eqref{Rproj1} that
\begin{equation}\label{term2}
\left|\pp_0^{d_i^*}((0,R))-c_2(p)(-1)^i\left[ e^{-\frac 2{p-1}(\zeta_i-\zeta_{i-1})}-e^{-\frac 2{p-1}(\zeta_{i+1}-\zeta_i)}\right] \right|\le CJ^{1+\delta_2}+CJ\bar J.
\end{equation}
Recall $d_i=-\tanh \zeta_i$ hence $\zeta_i'=-\frac{d_i'}{1-d_i^2}$. Using \eqref{2proj}, \eqref{term1}, \eqref{rough} and \eqref{term2}, we see that \eqref{est:zeta} is proved. 
Finally, we would like to stress the fact that since our computations are based on those appearing in the proof of \cite[Proposition 3.2]{MZajm11}, the constant $c_1(p)>0$ we get in \eqref{est:zeta} is the same as in that statement. 
\end{proof}

\section{Computation of the eigenvalues of  the matrix $M$ \eqref{defM}}\label{appmi}

We finish the proof of Lemma \ref{eigenm} here by proving formula \eqref{defmi}.Recall from \eqref{defsi} that
\[ \sigma_i = \frac{i(k-i)}{2}, \quad \text{for }  i=0, \dots, k. \]
For $j \in \llbracket 1, k \rrbracket$, we define the tridiagonal squared matrix of size $k-j+1$ by
\begin{equation} \label{mjk}
M_k^j = \begin{pmatrix}
- \sigma_j & \sigma_1 & \\
\sigma_j & -\sigma_1 - \sigma_{j+1} & \sigma_2 \\
0& \sigma_{j+1} & -\sigma_2 - \sigma_{j+1} & \sigma_3 \\
\\
& \ddots & & \ddots & \ddots \\
\\
& & & &   \sigma_{k-j+1} \\ 
& & & \sigma_{k-2}  & -\sigma_{k-j-1} - \sigma_{k-1} &  \sigma_{k-j} \\
& & & & \sigma_{k-1} &  -\sigma_{k-j}
\end{pmatrix}
\end{equation}

that is, on the upper diagonal lie $\sigma_1$, $\sigma_2$, \dots, $\sigma_{k-1}$, on the lower diagonal $\sigma_{j}$, $\sigma_{j+1}$, \dots, $\sigma_{k-1}$, and the sum of the coefficient of each \emph{column} is zero. Note that $M_k^k$ is the zero $1 \times 1$ matrix and that $M = M_k^1$. We now prove the following Lemma, which give the result for $j=1$.

\begin{lem}
$M_k^j$ is similar to the diagonal matrix 
\[ \mathrm{diag}(0, -j, - j - (j+1), - j - (j+1) - (j+2), \dots, - j - (j+1) - \cdots - (k-1)). \]
\end{lem}

Indeed, when $j=1$, we see that the eigenvalues of $M^1_k=M$ are given by $0$, $-1$, $-\frac{2\times 3}2$,\ldots,$-\frac{(k-1)\times k}2$, which is the desired conclusion announced in \eqref{defmi}. Let us prove the lemma now.
\begin{proof} 
We will prove the result by decreasing induction on $j$. The result is obvious for $j=k$. If it holds for $j+1$, let us prove it for $M_k^j$.
Let $A$ be the $(k-j+1) \times (k-j+1)$ matrix
\[ A = \begin{pmatrix}
1 & & 0 \\
\vdots & \ddots \\
1 & \cdots & 1 
\end{pmatrix}, \quad \text{so that} \quad A^{-1} =
\begin{pmatrix}
1 & 0 &  \cdots & & 0\\
-1 &  & \ddots &  & \vdots \\
0 & \ddots & \ddots &   \\
\vdots & \ddots & \ddots &  & 0 \\
0 & \cdots & 0 & -1 & 1 
\end{pmatrix}. \]
Let us conjugate $M_j^k$ by $A$ and compute:
\begin{align*}
A M_k^j A^{-1} & = \begin{pmatrix}
-\sigma_j & \sigma_1 & 0 & \cdots & 0\\
 0 & -\sigma_{j+1} & \sigma_2 & \ddots & \vdots\\
 \vdots & \ddots & \ddots & \ddots & 0 \\
 0 & \cdots & 0 & - \sigma_{k-1} & \sigma_{k-j} \\
0 & & \cdots & 0 & 0
\end{pmatrix} A^{-1} \\
& = \begin{pmatrix}
- \sigma_1 - \sigma_j & \sigma_1 & 0 & \cdots & 0 & 0 \\
\sigma_{j+1} & -\sigma_2  - \sigma_{j+1} & \sigma_2 & \ddots & \vdots & \vdots \\
0 & \sigma_{j+2} & -\sigma_3  - \sigma_{j+2} & \sigma_3 \\
\vdots & \ddots & \ddots & \ddots & 0 & 0 \\
& & & & \sigma_{k-j-1} & 0 \\
0 & \cdots & 0 & \sigma_{k-1} & -\sigma_{k-j} - \sigma_{k-1} & \sigma_{k-j}  \\
0 & & \cdots & 0 & 0 & 0
\end{pmatrix} \\
& = - j \Id_{k-j+1} + \begin{pmatrix}
& & & \vline & 0 \\
& & & \vline & \vdots \\
& M_k^{j+1} & & \vline & 0 \\
& & & \vline & \sigma_{k-1} \\
\hline
0 & \cdots & 0 & \vline &  j
\end{pmatrix},
\end{align*}
where the last line comes from the fact that
\begin{multline*}
-\sigma_1 - \sigma_j + \sigma_{j+1}  = \sigma_1 - \sigma_2  - \sigma_{j+1} + \sigma_{j+2} = \cdots \\
 = \sigma_{k-j-2} - \sigma_{k-j-1} - \sigma_{k-2} + \sigma_{k-1} = \sigma_{k-j-1} - \sigma_{k-j} - \sigma_{k-1} = -j
\end{multline*}
(because $\sigma_i$ is quadratic in $i$ with highest order term $-i^2/2$).

The induction hypothesis gives that the right-hand side block matrix is diagonalizable, with eingenvalues $j$, $0$, $- (j+1)$, $-((j+1) + (j+2))$, \dots, $-((j+1) + \cdots + (k-1))$. Hence $M^j_k$ is diagonalizable with eigenvalues $0$, $-j$, $-(j+ (j+1))$, $-(j+(j+1) + (j+2))$, \dots, $-(j+(j+1) + \cdots + (k-1))$. This concludes the induction and concludes the proof of Lemma \ref{eigenm} too.
\end{proof}

\frenchspacing
\bibliographystyle{plain}

\noindent{\bf Address}:\\
CMLS - École Polytechnique
91128 Palaiseau Cedex - France.\\
\vspace{-7mm}
\begin{verbatim}
e-mail: cote@math.polytechnique.fr 
\end{verbatim}
Universit\'e Paris 13, Institut Galil\'ee, 
Laboratoire Analyse, G\'eom\'etrie et Applications, CNRS UMR 7539,
99 avenue J.B. Cl\'ement, 93430 Villetaneuse, France.\\
\vspace{-7mm}
\begin{verbatim}
e-mail: Hatem.Zaag@univ-paris13.fr
\end{verbatim}

\end{document}